\documentclass[11pt,leqno,twoside]{amsart}
\usepackage{mathpazo}
\usepackage{geometry}
\usepackage{tikz-cd}
\usepackage{graphicx} 
\usepackage{amsmath}
\usepackage{amssymb}
\usepackage{extarrows}
\usepackage{csquotes}
\usepackage{comment}
\usepackage{mathrsfs}
\usepackage{amsthm}
\usepackage{enumitem}
\usepackage{color}
\usepackage{tikz}
\usepackage[colorlinks=true, linkcolor=blue, urlcolor=blue, citecolor=blue]{hyperref}
\newtheoremstyle{plainsl}
{8pt plus 2pt minus 4pt}
{8pt plus 2pt minus 4pt}
{\slshape}
{0pt}
{\bfseries}
{.}
{5pt plus 1pt minus 1pt}
{}
\theoremstyle{plainsl}
\newtheorem{theorem}{Theorem}[section]

\newtheorem{prop*}{Proposition}

\newtheorem{lemma}[theorem]{Lemma}
\newtheorem{corollary}[theorem]{Corollary}
\newtheorem{conjecture}[theorem]{Conjecture}
\newtheorem{remark}[theorem]{Remark}
\theoremstyle{definition}
\newtheorem{definition}[theorem]{Definition}
\newtheorem{example}[theorem]{Example}

\def\ba#1\ea{\begin{align*}#1\end{align*}}

\newcommand{\Q}{\mathbb{Q}}

\newcommand{\cH}{\mathcal{H}}

\newcommand{\cM}{\mathcal{M}}

\newcommand{\cO}{\mathcal{O}}

\newcommand{\bD}{\mathbf{D}}

\newcommand{\sD}{\mathscr{D}}

\newcommand{\gr}{\mathrm{Gr}}

\newcommand{\DR}{\mathrm{DR}}

\newcommand{\MHM}{\mathrm{MHM}}

\newcommand{\mhm}{\mathrm{MHM}}
\newcommand{\td}{\mathrm{td}}

\title[Hodge index theorem for finite quotient spaces]{On the conjectural Hodge index theorem for finite quotient spaces of singular varieties}
\author{Mohammadali Aligholi}
\address{M. Aligholi: Department of Mathematics, University of Wisconsin-Madison, 480 Lincoln Drive, Madison WI 53706-1388, USA}
\email{aligholi@wisc.edu}
\date{}
\begin{document}
\raggedbottom
\begin{abstract}
We study the conjectural characteristic class analogue of the Hodge index theorem for singular complex algebraic varieties, formulated by Brasselet--Sch\"urmann--Yokura, which expresses the Goresky--MacPherson homology $L$-classes in terms of suitable Hodge-theoretic $L$-classes, for  quotient spaces $X/G$, where $G$ is a finite group and $X$ is a pure-dimensional complex projective variety. Assuming that an equivariant $K$-theoretical version of the conjecture holds for $X$, we show that the characteristic class conjecture holds for $X/G$. Without assuming the equivariant $K$-theoretical version, we show that the characteristic class conjecture holds for $X/G$, provided that it holds for all fixed point sets $X^g$ and that a suitable normally nonsingular inclusion assumption is satisfied. Our approach is to work in equivariant analytic $K$-homology and equivariant algebraic $G$-theory and to identify the corresponding localized classes. On the analytic side, we identify the Banagl--Zagier equivariant $L$-classes with localized Chern character of the equivariant $K$-homology class of the  signature operator defined by Banagl--Leichtnam--Piazza. On the algebraic side, we compute the localization of the equivariant motivic Hodge-Chern class transformation of the intersection Hodge module. 
\end{abstract}
\setcounter{tocdepth}{3}
\maketitle
\tableofcontents
\section{Introduction}
Brasselet–Sch\"urmann–Yokura \cite{BSY} formulated a conjectural characteristic class analogue of the Hodge index theorem for singular complex algebraic varieties. More precisely:
\begin{conjecture} \label{BSYi}
Let $X$ be a compact pure-dimensional complex algebraic variety. Then
\begin{equation}\label{conji}
L_*(X)=IT_{1*}(X)
\end{equation}
\end{conjecture}
Here 
\begin{equation}
IT_{y*}(X):=MHT_{y*}([IC'^H_X]) \in H_*(X) \otimes \Q[y^{\pm 1}],
\end{equation}
with $IC'^H_X:=IC_X^H[-\dim_{\mathbb{C}}(X)]$.   Let $\MHM(X)$ denote the abelian category of algebraic mixed Hodge modules on $X$ \cite{Sa}. The Hodge-theoretic {\it Hirzebruch class transformation} \cite{BSY} is defined by the composition
\begin{align*}
MHT_{y\ast}: K_0(\mhm(X)) &\to H_*(X) \otimes \Q[y^{\pm 1}, (1+y)^{-1}] \\
MHT_{y\ast} &=\td_{(1+y)\ast}  \circ MHC_y
\end{align*}
where $\td_{(1+y)\ast}$
is the scalar extension of the
Baum--Fulton-MacPherson Todd class transformation $\td_*:G_0(X) \to H_\ast(X) \otimes \Q$ \cite{BFM}, together with the multiplication by $(1+y)^{-k}$ on $H_{2k}(X)$,  and a  motivic  Chern class transformation
\[
 MHC_y: K_0(\MHM(X)) \to  G_0(X)[y,y^{-1}]
\]
where $G_0(X)=K_0(Coh(X))$ is the Grothendieck group of algebraic coherent sheaves
of $\cO_X$-modules on $X$. 
If $X$ is smooth and $M\in \MHM(X)$, let $(\cM,F_{\bullet}\cM)$ be the underlying filtered $\sD_X$-module of $M$, and set
\[
\begin{split}
MHC_y[M] &= \sum_{p}  \left[\gr^F_{-p}\DR(\cM)\right] \cdot(-y)^p  \\
&= \sum_{p,i} (-1)^i \left[\cH^i \gr^F_{-p}\DR(\cM)\right] \cdot(-y)^p \in G_0(X)[y,y^{-1}]
\end{split}
\]
with $\gr^F_{-p}\DR(\cM)$ the graded parts of the de Rham complex of $\cM$ with respect to the induced Hodge filtration. 

 The Goresky-MacPherson $L$-class \cite{GM} is given by $$L_*(X)= L_*(\left[IC_X,\alpha_{can}\right])$$ 
 where $IC_X$ is the intersection complex of $X$ equipped with the self-duality isomorphism
\[
\alpha_{\mathrm{can}}: IC_X \xrightarrow{\cong} \mathbb{D}_X IC_X
\]
coming from Verdier duality. $L_*$ is a functorial $L$-class transformation
 \[L_*:\Omega_{\mathbb{R}}(X)\rightarrow H_*(X,\mathbb{R})\]
  from the cobordism group  of self-dual sheaf complexes $\Omega_{\mathbb{R}}(X)$, constructed via the methods of Cappell-Shaneson \cite{CS}, and its relation to other  $L$-class transformations was explained in detail in \cite[Section~3.3]{AMS}.  
  
Conjecture \ref{BSYi} holds in the smooth context by work of Hirzebruch on the generalized Hirzebruch-Riemann-Roch theorem.  The conjecture has been proved in many cases, including  the case of rational homology manifolds in \cite{FP,FPS, AMS}, all  compact toric varieties and
(matroid) Schubert varieties, complex algebraic surfaces and threefolds in \cite{AMS}.  In this paper we focus on proving the conjecture for a quotient space $X/G$, $G$ a finite group.

Before we discuss the case of the quotient space $X/G$, we say a few words on why it's more natural to work in $K$-homology instead of $H_*(X,\mathbb{C})$. Conjecture \ref{BSYi} concerns the coincidence of two homology classes in $H_*(X,\mathbb{C})$. We can lift this problem to the level of $K$-homology. On the algebraic side, by definition
\[
IT_{1*}(X)=\td_{2*}(MHC_1(IC'^H_X))
\]
where $MHC_1(IC'^H_X) \in G_0(X)$. On the analytic side, the $K$-homology class of signature operator, whose   Chern character computes (renormalized) Goresky--MacPherson $L$-class,  was constructed by Moscovici--Wu \cite{MW} for Witt spaces, using results of Cheeger \cite{Ch}, for any metric on the regular part quasi-isometric to a piecewise-flat conic metric. Albin--Leichtnam--Mazzeo--Piazza \cite{ALMP} reconstructed the signature operator using a \textit{wedge} (iterated edge) metric, which belongs to the class considered by \cite{MW}. ($n=dim_{\mathbb{R}}X$ is even as $X$ is a complex variety)

\[
\left[D^{sign}_X\right]\in K^{an}_n(X)\cong K^{an}_0(X)
\]
\[\Psi_2Ch(\left[D^{sign}_X\right])=L_*(X)\]
 Thus, instead of trying to compare the two homology classes directly, which is difficult because of the purely topological construction of $L$-classes, we may instead compare the corresponding $K$-theoretic classes. We briefly recall how to bridge various $K$-homologies for a complex projective variety $X$.  There are comparison  maps between topological complex $K$-homology,  Kasparov $K$-homology,  and  Grothendieck group of algebraic coherent sheaves:
\[
G_0(X)\xrightarrow{\alpha} K_0^{top}(X)
\xrightarrow{\underset{\cong}{\lambda_c}}
K_0^{an}(X)
\]
$\alpha$ is the Baum-Fulton-MacPherson Riemann-Roch transformation \cite{BFM2}. See Section \ref{compairsonmaps} for definition of $\lambda_c$. We have the compatibility of  Chern characters and Todd class transformation which make the following diagram commutative:
\begin{equation}\label{compatiblitcherntodd}
\begin{tikzcd}
K_0^{an}(X)
 \arrow[drrr, "Ch"]
\\
K_0^{top}(X)
    \arrow[rrr, "ch^{top}"]
    \arrow[u, "\lambda_c"]
    &&&
H_*(X,\mathbb{C})
\\
G_0(X)
    \arrow[u, "\alpha"]
    \arrow[urrr, "\td_*"]
\end{tikzcd}
\end{equation}
The Sullivan–Siegel orientation of $X$ is a real $K$-homology class 
\[\Delta(X)\in KO_{n}^{top}(X)[\frac{1}{2}]\]
It was shown by P. Albin, M. Banagl, and P. Piazza \cite[Thm. 7.6]{ABP1} that 
\[\lambda_c \circ c \circ (\Psi_{KO}^2)^{-1} (\Delta(X))=sign_K(X):=2^{-n/2}\left[D_X^{sign}\right]\in K_{n}^{an}(X)[\frac{1}{2}]\cong K_{0}^{an}(X)[\frac{1}{2}]\]
Where 
\[c:KO_{n}^{top}(X)[\frac{1}{2}]\rightarrow K_{n}^{top}(X)[\frac{1}{2}]\]
is the complexification map and
\[\Psi_{KO}^2:KO[\frac{1}{2}]\rightarrow KO[\frac{1}{2}]\]
is the stable Adams operation on the $KO$-spectrum.
We may consider a stronger conjecture which provides the corresponding object in $G_0(X)$.
\begin{conjecture}($K$-theoretical Hodge index conjecture)\label{k-conjecture}
    Let $X$ be a  pure-dimensional projective complex algebraic variety. Then the following equality holds in $K_0^{an}(X)$
    \begin{equation}
        (\lambda_c \circ \alpha)([MHC_{1}(IC'^H_X)])=\left[D^{sign}_X\right]
    \end{equation}
\end{conjecture}
This is by no means an easier problem, but working in $K$-theory has several benefits. For example, $L$-classes are global, and finding a local-to-global formula for even the signature is a notoriously difficult problem. J. Cheeger \cite[Thm. 9.3]{Ch} found such a formula for a closed oriented Riemannian pseudomanifold with a piecewise-flat metric in terms of a sum of local $\eta$-invariants associated with the links of strata, by the heat equation methods. However, these invariants are  difficult to compute explicitly and may well not be rational numbers in general, and have no known topological interpretation. For a general smooth non-compact manifold, there is no ordinary closed-manifold signature, whereas the signature operator still defines a Dirac-type operator. In $K$-theory, there are long exact sequences relating classes supported on open and closed subspaces, and hence one can break down the class of the signature operator.

Another benefit of working in $K$-theory arises in the equivariant setting where the action by a group is available. Equivariant characteristic classes are  classes often supported on fixed point sets and can be viewed as arising from  localization at the level of $K$-theory. Then these equivariant classes could be used to express characteristic classes of the quotient space $X/G$.  When $X$ is a  projective $G$-manifold  and $G$ is a finite group of algebraic automorphisms, $X/G$ is a rational homology manifold; hence, as mentioned above, Conjecture \ref{BSYi} holds for $X/G$. Nevertheless  S. E. Cappell, L. Maxim, J. Sch\"urmann, and J. L. Shaneson \cite{CMSS} proved the conjecture for $X/G$   directly by  expressing the $L$-class and Hirzebruch class of $X/G$ as averages of localized classes on fixed point sets.
\begin{theorem} \cite[Corollary 5.7]{CMSS} \label{averagetoddalgebraic} Let $X$ be a pure-dimensional complex quasi-projective variety acted upon by a
finite group $G$ of algebraic automorphisms.
     \begin{equation}\label{introit1fomrula}
         IT_{1*}(X/G)=\frac{1}{|G|}\sum_{g\in G} \pi_*^gIT_{1*}(X,g)
     \end{equation}
where $\pi^g:X^g\xrightarrow{i} X\xrightarrow{\pi}X/G$.
\end{theorem}
$IT_{1*}(X,g)$ is obtained from  \cite{CMSS}:
\[IT_{1*}(X,g):=\td_{2*}(g)( MHC^G_1 (IC'^H_X))\]
where
\[
 MHC^G_y: K_0(\MHM^G(X)) \to  G^G_0(X)[y,y^{-1}]
\]
is the G-equivariant motivic Chern class transformation, defined in \cite{CMSS}, from the Grothendieck group of $G$-equivariant algebraic mixed Hodge modules to the Grothendieck group of $G$-equivariant algebraic coherent sheaves on $X$. We mention that the construction of $MHC^G_y$ in \cite{CMSS} relies heavily on their notion of \textit{weak equivariant derived categories}  to adapt Saito’s
functors $\gr^F_{-p}\DR$ to the equivariant context. We have $\td_{2*}(g)=\Psi_2\circ \td_{*}(g)$ where
\[ \td_{*}(g)(-): G^G_0(X)\rightarrow H_*^{BM}(X^g,\mathbb{C})\]
is Lefschetz–Riemann–Roch 
transformation \cite{BFQ, CMSS}. Note that the (shifted) intersection Hodge module $IC'^H_X$ is $G$-equivariant. It follows that $IT_{1*}(X,g)$ is supported on the fixed point set $X^g$. 
One can view $\td_{*}(g)$ as Todd class transformation applied to the algebraic localization in $G$-theory;
\[IT_{1*}(X,g)=\Psi_2\td_{*}(MHC^G_1 (IC'^H_X),g)\]
where the transformation $\td_{*}(-,g)$ is the composition of the maps ($G'=\langle g \rangle$)
    \[G_0^{G}(X)\xrightarrow{res^G_{G'}} G_0^{G'}(X)\rightarrow G_0^{G'}(X)_{{g}}\xrightarrow[]{L_{{g}}}G_0^{G'}(X^g)_{g}\simeq G_0(X^g) \otimes R(G')_{g}\xrightarrow{\td_{*}\otimes \chi_g}  H_*(X^g,\mathbb{C})\]
 where the algebraic localization map $L_g$ is the inverse of the isomorphism induced from the inclusion $i:X^g\hookrightarrow X$ \cite{T1, BFQ, Q79, Ta} 
     \[
L_g=(i_*)^{-1}:G_0^{G'}(X)_g\xrightarrow{\cong}G_0^{G'}(X^g)_g
\]
The analogue of (\ref{introit1fomrula}) for $L_*(X/G)$ was proved by D. Zagier \cite{Z} in the case of smooth $X$. A generalization to Witt pseudomanifolds was obtained by  M. Banagl \cite{B26}.
\begin{theorem} \cite[Thm. 10.1]{B26} \label{banaglintro}
   Let $G$ be a finite group and $X$ an oriented closed Witt pseudomanifold upon
which $G$ acts simplicially, preserving the orientation. Then $X/G$ is a Witt space and
    \begin{equation}\label{averagelclass}
        L_*(X/G)=\frac{1}{|G|}\sum_{g\in G} \pi_*L_*(X,g)
    \end{equation}
    where $\pi: X\rightarrow X/G$ is the projection, and $L_*(X,g)\in H_*(X,\mathbb{C})^G$ are Banagl's \textit{equivariant $L$-classes of Atiyah–Singer–Zagier type}.
\end{theorem}
 We prove, in Theorem \ref{Banagl_localization_comparison}, that $L_*(X,g)$ coincides (after normalization)  with the localized Chern character of the analytic equivariant $K$-homology signature class $[D^{sign,G}_{X}]$ of \cite{BLP}, thereby clarifying the relationship between $L_*(X,g)$ and $[D^{sign,G}_{X}]$ mentioned in \cite[p.~4]{B26}. In particular, one can see that $L_*(X,g)$ is supported on the fixed point set $X^g$, which is not clear from the definition of $L_*(X,g)$ given  in loc. cit.
\begin{theorem}\label{bangltolocalizeintro}
Let $G$ be a finite group and let $X$ be a compact oriented even-dimensional triangulated Witt pseudomanifold on which $G$ acts by orientation-preserving simplicial maps. Let $g\in G$, and let
$
i:X^g\rightarrow X
$
be the inclusion of the fixed point set. Then
\[
L_*(X,g)=i_*(\Psi_2(L^{an}_*(X,g)))
\]
\[
L^{an}_*(X,g):=Ch(\left[D_X^{sign,G}\right],g)
\]
 where the transformation $Ch(-,g)$ is the composition of the maps 
    \[K_0^{G,an}(X)\xrightarrow{res^G_{G'}} K_0^{G',an}(X)\rightarrow K_0^{G',an}(X)_{{g}}\xrightarrow[]{\gamma_{{g}}}K_0^{G',an}(X^g)_{g}\simeq K_0^{an}(X^g) \otimes R(G')_{g}\xrightarrow{Ch\otimes \chi_g}H_*(X^g,\mathbb{C})\]
here $\gamma_g$ is the localization map in analytic $K$-homology \cite{BLP} for the action of $G'=\langle g \rangle$
\[\gamma_{{g}}=(i_*)^{-1}:K_0^{G',an}(X)_{{g}}\xrightarrow[]{\cong}K_0^{G',an}(X^g)_{g}\]
\end{theorem}
In the smooth case, comparison of $L_*(X,g)=i_*IT_{1*}(X,g)$ follows immediately from comparison of cohomological characteristic classes as in proving $L_*(X)=IT_{1*}(X)$. In the singular case, the problem is how to compare them. Since $IT_{1*}(X,g)$ and $L^{an}_*(X,g)$ are  obtained from the localization of $MHC^G_1(IC'^H_X)$
and $[D_X^{sign,G}]$, there are two approaches. First, we can assume the equivariant version of  Conjecture \ref{k-conjecture} holds for $X$, and then show Conjecture \ref{BSYi} holds for the quotient space $X/G$.
\begin{conjecture} \label{G-K-CONJECTURE}(Equivariant $K$-theoretical Hodge index conjecture)
    Let $G$ be a finite group acting algebraically on a  pure-dimensional projective complex algebraic variety $X$. Then the following equality holds in $K_0^{G,an}(X)$
    \[
(\lambda^G_c\circ \alpha^G)([MHC^G_{1}(IC'^H_X)])=\left[D^{sign,G}_X\right]
\]
\end{conjecture}
where $\lambda^G_c,\alpha^G$ are equivariant versions of comparison maps $\lambda_c$ and $\alpha$, explained in section (\ref{equivcomparison}). Naturality of maps implies
\[
\begin{tikzcd}[column sep=large, row sep=large]
G_0^{G}(X)
    \arrow[r, "\operatorname{res}^{G}_{G'}"]
    \arrow[d, "\lambda^G_c\circ \alpha^G"']
&
G_0^{G'}(X)
    \arrow[r]
    \arrow[d, "\lambda^{G'}_c\circ \alpha^{G'}"']
&
G_0^{G'}(X)_{g}
    \arrow[r, "L_{g}"]
    \arrow[d, "(\lambda^{G'}_c\circ \alpha^{G'})_g"']
&
G_0^{G'}(X^{g})_{g}
    \arrow[r, "\td_{*}\otimes\chi_g"]
    \arrow[d, "(\lambda_c\circ \alpha)\otimes 1"']
&
H_*(X^{g},\mathbb C)
    \arrow[d, equal]
\\
K_0^{G,\mathrm{an}}(X)
    \arrow[r, "\operatorname{res}^{G}_{G'}"']
&
K_0^{G',\mathrm{an}}(X)
    \arrow[r]
&
K_0^{G',\mathrm{an}}(X)_{g}
    \arrow[r, "\gamma_{g}"']
&
K_0^{G',\mathrm{an}}(X^{g})_{g}
    \arrow[r, "Ch\otimes\chi_g"']
&
H_*(X^{g},\mathbb C)
\end{tikzcd}
\]
is commutative. Therefore, assuming Conjecture \ref{G-K-CONJECTURE}, Theorems \ref{averagetoddalgebraic}, \ref{banaglintro}, and \ref{bangltolocalizeintro} imply
\begin{theorem}
     Let $G$ be a finite group acting by algebraic automorphisms on a   pure-dimensional complex projective variety $X$. Assume the equivariant $K$-theoretical Hodge index conjecture \ref{G-K-CONJECTURE} holds for $X$. Then Conjecture \ref{BSYi} holds for the quotient space $X/G$.
\end{theorem}
Second, without assuming the equivariant K-theoretical version, we may instead assume that  Conjecture \ref{BSYi}  holds for all  fixed point sets $X^g$ and that the inclusions $X^g\hookrightarrow X$ satisfy a normally nonsingular type condition called \textit{nice} (Definition \ref{nice}).
 \begin{theorem}\label{conditionalconjectureideniicationintroduction}
     Let $G$ be a finite group acting by algebraic automorphisms on a   pure-dimensional complex projective variety $X$. For every $g\in G$ assume the following: For $G'=\langle g\rangle$, the inclusion of the fixed point set $X^g\hookrightarrow X$ is  \textit{G'-nicely} normally nonsingular, and  Conjecture \ref{BSYi} holds for $X^g$. Then Conjecture \ref{BSYi} holds for the quotient space $X/G$.
 \end{theorem}
This is  Corollary \ref{theoremorbitspace}. We make use of recent results from \cite{BLP} to describe $L^{an}_*(X,g)$ under the assumption that the inclusions of the fixed-point sets are normally nonsingular. An explicit formula for the localized signature class is obtained there, following \cite{R}, only under this assumption. A stronger condition than topological normal nonsingularity, which is the $G'$-nicely normally nonsingular, involving compatibility between the algebraic and topological normal bundles together with a suitable transversality condition, is needed to compute the algebraic localized Hirzebruch classes $IT_{1*}(X,g)$ (Corollary \ref{tdgformula}).
\begin{theorem}
Let $G=\langle g\rangle$ be a finite cyclic group acting on  a  pure-dimensional complex projective variety $X$ by
algebraic automorphisms. If  the inclusion of  fixed point set $X^g\hookrightarrow X$ is  \textit{G-nicely} normally nonsingular with conormal bundle $N^{\vee}_{X^g/X}$, then
     \[L_g(MHC^G_1(IC'^H_X)_g)=\lambda_1(N^{\vee}_{X^g/ X})_g\otimes(\lambda_{-1}N^{\vee}_{X^g/X})_g^{-1}\otimes  MHC^G_1(IC'^H_{X^g})_g \]
\end{theorem}
 This follows from the same arguments as  Banagl's \cite{B}.

\medskip
{\bf Acknowledgments.} I deeply thank Markus Banagl for patiently answering my questions. I thank Eric Leichtnam for helpful discussions. I am especially grateful to my advisor, Lauren\c{t}iu Maxim, for his continued support and many helpful discussions.
\section{Conventions and Notation}\label{convennotaion}
\begin{enumerate}[label=(\alph*)]
    \item For simplicity of notation we assume all fixed point sets $X^g$ are connected. For disconnected $X^g$ the formulas involving localization should be interpreted as sums over their connected components, using the dimension of each component. We assume all real manifolds are even-dimensional. Since our ultimate interest is in complex algebraic varieties, this is not a restriction.
    
\item For a finite group $G$, $R(G)$ denotes the complex representation ring of $G$. For an $R(G)$-module $M$, and $g\in G$, we denote by $M_g$ the localization of $M$ with respect to the prime ideal of $R(G)$ consisting of virtual representations whose character vanishes at $g$. Let
$$\chi_g: R(G)_g \rightarrow \mathbb{C}\qquad \frac{\psi}{\zeta} \mapsto\frac{\psi(g)}{\zeta(g)}$$
\item \label{varietysatisfyconditions} Throughout the paper, we may apply results of \cite{BLP} and \cite{B26} to a compact pure-dimensional complex algebraic variety $X$ equipped with an algebraic action of a finite group $G$. $X$ is an oriented Witt space. By \cite[Prop. 6.7]{BLP}, subanalytic proper actions admit a $G$-equivariant triangulation.   As explained in \cite[Prop. 9.1]{B26}, $X$ can be smoothly stratified, in the sense of \cite{ABP1} by taking the open simplices as strata and $G$ acts by stratified diffeomorphisms. $X$  admits control data giving it the structure of a $G$-Thom-Mather  space, in the sense of \cite[Def. 2.1]{BLP}. Hence the hypotheses needed to apply the analytic results of \cite{BLP} and the simplicial results of \cite{B26} are satisfied.

\item \label{notationconvention}
 $K^{top}_{0}, K^{an}_{0}, K_0^{geo}, G_0$  denote, respectively,  topological complex $K$-homology,  Kasparov $K$-homology $K^{an}_{0}(X)=KK_{0}(C_0(X),\mathbb{C})$ \cite{K},  Baum--Douglas $K$-homology \cite{BD}, and  $G_0(X)=K_0(Coh(X))$ is the Grothendieck group of algebraic coherent sheaves.
 
\item Let $n=d_{\mathbb{R}}X$. In this paper, $n$ is even, hence $K_0\cong K_n$. We consistently use only $K_0$ in this paper.

\item We follow the notation used in \cite{CMSS} for equivariant mixed Hodge modules. The category $\mhm^G(X)$ is the abelian category of $G$-equivariant algebraic mixed Hodge modules on $X$, and $K_0(\mhm^G(X))$ denotes its Grothendieck group. Also $G^G_0(X)=K_0(Coh^G(X))$ denotes the Grothendieck group of the abelian category of $G$-equivariant coherent algebraic sheaves on $X$.  See also Appendix \ref{appendix}.

\item $n$-th homological Adams operation is defined by $\Psi_n|_{H_{2k}}=n^{-k}\operatorname{id}$. Similarly $\Psi^n$ denotes $n$-th cohomological Adams operation $\Psi^n|_{H^{2k}}=n^k\operatorname{id}$.
The subscript $(2)$ in a homological transformation, such as $Ch_{(2)}$ or $\td_{2*}$, denotes composition with the second homological Adams operation $\Psi_2$, i.e., $Ch_{(2)}:=\Psi_2\circ Ch$ and $\td_{2*}:=\Psi_2\circ\td_*$. Likewise,  subscript $(2)$ in a cohomological transformation such as Chern character of vector bundles $ch_{(2)}$ denotes composition with the second cohomological Adams operation. The subscript $(g)$ denotes composition with the evaluation map $\chi_g:R(G)_g\rightarrow\mathbb{C}$, e.g.,
$$Ch_g:=Ch\otimes\chi_g:K_0^{an}(X)\otimes R(G)_g\rightarrow H_*(X,\mathbb{C})$$
\end{enumerate}
\section{Comparison of algebraic, topological, and analytic $K$-homology}
\label{compairsonmaps}
    Let $X$ be a complex projective algebraic variety.  We have the comparison  maps 
\[
G_0(X)\xrightarrow{\alpha} K_0^{top}(X)
\xleftarrow{\underset{\cong}{\phi}}
K_0^{geo}(X)
\xrightarrow{\underset{\cong}{\mu}}
K_0^{an}(X)
\]
where 
\[\alpha:G_0(X) \rightarrow K_0^{top}(X) \]
is the Baum-Fulton-MacPherson Riemann-Roch transformation \cite{BFM2}, where for nonsingular $X$, $\alpha(\mathcal{O}_X)=[X]_K$ is the topological $K$-homology fundamental class of $X$. $\alpha$ is covariant for proper morphisms, and is compatible with cap
products, exterior products, and restrictions.

A class $[(M,E,f)]\in K_0^{geo}(X)$ is represented by a triple where $M$ is a closed even-dimensional $spin^c$ manifold, $E\rightarrow M$ is a complex vector bundle, and $f:M\rightarrow X$ is a continuous map. Equivalence of geometric and analytic $K$-Homology \cite{BHS} is given by 
\[\mu:K_0^{geo}(X) \xrightarrow{\cong} K_0^{an}(X) \]
\[[(M,E,f)]\mapsto f_*[\bD_M^E]\] 
where $f_*:K_0^{an}(M) \rightarrow K_0^{an}(X)$ and $[\bD_M^E]$ denotes the  class in $K_0^{an}(M)$ corresponding to the  $spin^c$ operator on $M$ twisted by $E$. 
\[[\bD_M^E]:=E\cap[\bD_M^{spin^c}]:=[E]\otimes_{C(M)}[\bD_M^{spin^c}] \in K_0^{an}(M)\]
In $KK$-theory, $E$ determines an element
\[
[E]\in KK_0(C(M),C(M))
\]
There is an isomorphism 
\[\phi:K_0^{geo}(X) \xrightarrow{\cong} K_0^{top}(X) \]
\[[(M,E,f)]\mapsto f_*([E]\cap [M]_K)\]
where $\cap$ denotes the $K^0$-module structure of $K^{top}_0$; $\cap:K^* \otimes K_*\rightarrow K_*$. 

Since we will use various results of \cite{ABP2}, let us express  $\phi$ in terms of comparison maps given there. \cite{ABP2} uses  a bordism-like description of $K$-homology due to Jakob \cite{J}, denoted by $K^b_0(X)$ with an isomorphism \cite[\S 4.1]{ABP2}
\[\varphi: K^b_0(X)\xrightarrow{\cong} K_0^{top}(X)\]
and an 
isomorphism \cite[\S 4.2]{ABP2}
\[\kappa: K^b_0(X)\xrightarrow{\cong} K_0^{an}(X)\]
then we have
\[\phi=\varphi\circ \kappa^{-1}\circ\mu\]
in the notation of \cite[Thm. 6.6]{ABP2}, we set 
\begin{equation}\label{lambda-cabp2}
    \lambda_c:=\mu \circ \phi^{-1} :K_0^{top}(X) \rightarrow K_0^{an}(X)
\end{equation}
We denote the Chern character on vector bundles on a manifold $M$ as
\[ch:K^0(M)\rightarrow H^*(M,\mathbb{C})\] and the homological Chern character \cite{BFM2} as
\[ch^{top}:K_0^{top}(X)\rightarrow H_*(X,\mathbb{C})\]
The Chern character on $K_0^{geo}(X)$ is defined by \cite[\S~2.6]{BE}
\[ch^{geo}(M,E,f)=f_*(ch(E)\cup Td^{spin^c}(TM)\cap [M])\]
\[Td^{spin^c}(TM)=e^{c_1(L_{M})/2}\hat{A}(TM)\]
where $L_{M}$ is the determinant line bundle of the $spin^c$ structure. We may define the analytic Chern character by  \cite[\S~2.6]{BE}
\[Ch:=ch^{geo}\circ\mu^{-1}:K_0^{an}(X) \rightarrow  H_*(X,\mathbb{C})\]
There is the Baum--Fulton--MacPherson Todd class transformation \cite{BFM}, 
\[\td_*:G_0(X) \rightarrow  H_*(X,\mathbb{C})\]
such that for a non-singular $X$, $\td_*(\mathcal{O}_X)=Td(X)\cap [X]$.

 We have the compatibility of topological, analytic Chern characters, and Todd  class transformation;
\begin{equation}\label{compatiblitcherntodd2}
\begin{tikzcd}
K_0^{geo}(X)
    \arrow[d, "\phi"]
    \arrow[rrr, "\mu"]
    \arrow[ddrrr, "ch^{geo}"]
&&&
K_0^{an}(X)
    \arrow[dd, "Ch"]
\\
K_0^{top}(X)
    \arrow[drrr, "ch^{top}"]
\\
G_0(X)
    \arrow[u, "\alpha"]
    \arrow[rrr, "\td_*"]
&&&
H_*(X,\mathbb{C})
\end{tikzcd}
\end{equation}
 For  $\td_*= ch^{top}\circ \alpha$ see \cite[\S~5]{BFM2}, and for $ch^{geo}=ch^{top}\circ \phi$ see \cite{ABP2,J}. Note that the normalization we use  for $Ch$ is such that 
 \begin{equation}\label{ch=mathcalL}
     Ch(\left[D_X^{sign}\right])=L^{an}_*(X)
 \end{equation}
 where $\Psi_2L^{an}_*(X)=L_*(X)$ is the Goresky--MacPherson $L$-class and  $\left[D_X^{sign}\right] \in K_0^{an}(X)$ is defined in \cite{MW,ALMP}. This is the convention used in \cite{BLP} and one that makes (\ref{compatiblitcherntodd2}) commutative.
\subsection{Equivariant comparison maps}\label{equivcomparison}
Let $G=\langle g\rangle$ be a finite cyclic group acting on a complex projective algebraic variety $X$ by algebraic automorphisms. In this paper we only need cyclic groups, as we will mention, the equivariant localized classes only depend on the cyclic subgroup $\langle g\rangle$, hence we'll simply restrict to this subgroup. Denote  the $G$-equivariant  $K$-homologies by $K^{G,top}_{0}(X), K^{G,an}_{0}(X):=KK^G_0(C_0(X),\mathbb{C}), K_0^{G,geo}(X), G^G_0(X)$.  There are equivariant versions of comparison maps. These are natural maps satisfying obvious functoriality in their equivariant category.

The equivariant version of $\mu$ was studied in \cite{BOSW}.
\begin{theorem}\cite{BOSW} \label{BOSW}
    There is a natural isomorphism
    \[\mu^G:K_0^{G,geo}(X) \xrightarrow{\cong} K_0^{G,an}(X) \]
\end{theorem}
A class in $K_0^{G,geo}(X)$ is a triple $[(M,E,f)]$ where $M$ is a $G$-$spin^c$ manifold, $E$ is a complex $G$-vector bundle on $M$, and $f:M\rightarrow X$ is $G$-equivariant.

We can consider the equivariant version of $\phi$:
  \[\phi^G:K_0^{G,geo}(X) \rightarrow K_0^{G,top}(X) \]
\[[(M,E,f)]\mapsto f_*([E]\cap [M]_K)\]
which makes sense since $E\in K^0_G(M), [M]_K\in K^{G,top}_0(M)$. As
 $K_0^{G,an}(X)$ is naturally isomorphic to  $K_0^{G,top}(X)$ \cite{Sch}, and $\mu^G$ is naturally  isomorphic, $\phi^G$ is also isomorphism. Set
 \[ \lambda^G_c=\mu^G\circ(\phi^G)^{-1}:K_0^{G,top}(X) \rightarrow K_0^{G,an}(X)\]
As explained in \cite[\S~5]{T2}, the constructions of \cite{BFM2} are all
sufficiently natural to yield the equivariant analogs without change. More precisely, for a possibly singular complex $G$-variety $X$, we can choose a $G$-equivariant closed embedding $j:X\hookrightarrow M$, with $M$ being a smooth $G$-variety. We then use the equivariant relative algebraic $K$-group $K^{G,alg}_X(M)$ defined from bounded complexes of $G$-equivariant algebraic vector bundles on $M$ that are exact on $M\setminus X$, then
\[K^{G,alg}_X(M)\cong G_0^G(X)\]
Algebraic-to-topological comparison for equivariant vector-bundle complexes gives a forgetful map
\[K^{G,alg}_X(M) \rightarrow K^{G,top}_X(M)\]
Finally, an equivariant Alexander duality for the smooth complex $G$-manifold $M$ gives
\[K^{G,top}_X(M) \cong K_0^{G,top}(X)\]
\begin{theorem}\cite{T2}
    There is a natural map
    \[\alpha^G:G_0^{G}(X) \rightarrow K_0^{G,top}(X) \]
\end{theorem}
\subsection{Localization in $K_0^{G,top}, K_0^{G,an}$, and $G_0^G$}
There are $K$-homology localization maps in topological, analytic, and algebraic $K$-homology, originally studied by Atiyah and Segal  \cite{AS2} in the context of Lefschetz fixed-point theorems on manifolds. The localization theorem in topological $K$-homology was proved by Rosenberg--Weinberger \cite[Thm.~2.4]{RW},  the localization theorem in analytic $K$-homology was proved by Rosenberg and Banagl--Leichtnam--Piazza \cite[Prop.~6.6]{R, BLP}, and the localization theorems in $G_0$-theory of algebraic coherent sheaves were studied by Baum--Fulton--Quart, Thomason, Takeda \cite{BFQ,Q79, T1,Ta}.

Let $G=\langle g\rangle$ be a finite cyclic group. For the sake of notation, we may assume the fixed point set $X^g$ is connected. This is not a restriction as the localization formula will involve a sum over the connected components. Let $i:X^g\rightarrow X$ be the inclusion.
\begin{theorem} \cite[Thm. 2.4]{RW}  Let $X$ be a compact $G$-CW-complex. The inclusion  $i:X^g\rightarrow X$ induces an isomorphism    $$(i^{top})_*:K_0^{G,top}(X^g)_{g}\xrightarrow{\cong} K_0^{G,top}(X)_{{g}}$$ 
\end{theorem}
Let the inverse of $(i^{top})_*$ be
\[\delta_g:=(i^{top})_*^{-1}:K_0^{G,top}(X)_{{g}}\xrightarrow[]{\cong}K_0^{G,top}(X^g)_{g}\]
\begin{theorem} \label{inverseoflocalization}\cite{R}, \cite[Prop. 6.6]{BLP}  Let $X$ be a compact $G$-CW-complex. There is an isomorphism    $$(i^{an})_*:K_0^{G,an}(X^g)_{g}\xrightarrow{\cong} K_0^{G,an}(X)_{{g}}$$ 
defined by Kasparov product from the left with $[i^*]\in KK^{G}_0(C(X),C(X^g))$, a Kasparov bimodule associated with the homomorphism of $C^*$-algebras $C(X)\rightarrow C(X^g)$ induced from the restriction.
\end{theorem}
Let the inverse of $(i^{an})_*$ be
\[\gamma_{{g}}:=(i^{an})_*^{-1}:K_0^{G,an}(X)_{{g}}\xrightarrow[]{\cong}K_0^{G,an}(X^g)_{g}\]
\begin{remark}\label{cwlocalizationblp}
   In \cite[Prop. 6.6]{BLP}, the localization is defined for any compact Lie group $G$. By \cite[Prop. 6.7]{BLP}, if $X$ is a locally compact subanalytic set and $G$ is a subanalytic proper transformation group of $X$, then $X$ admits  a $G$-equivariant triangulation. Finite groups are subanalytic groups. When every orbit is $G-spin^c$, the localization map is an isomorphism. This condition is satisfied when $G$ is abelian. In this paper, all analytic localization maps are taken with respect to a finite cyclic group.
\end{remark}
Now we discuss algebraic localization in $G$-theory. Although the localization is available in general for a finite type scheme, we restrict to the quasi-projective case in order to  apply  Lefschetz–Riemann–Roch transformation of \cite{BFQ} later on.
\begin{theorem}(\cite[Thm.~ 1.1]{Ta}, \cite{Q79}) \label{coherencetraceformula1} Let $X$ be a quasi-projective complex algebraic variety, with an algebraic action by a finite cyclic group $G=\langle g\rangle$.  The inclusion  $i:X^g\rightarrow X$ induces an isomorphism 
$$(i^{alg})_*:G_0^{G}(X^g)_{g}\xrightarrow{\cong} G_0^{G}(X)_{{g}}$$ 
\end{theorem}
Let the inverse of $(i^{alg})_*$ be
\[
L_g:=(i^{alg})_*^{-1}:G_0^G(X)_g\xrightarrow{\cong}G_0^G(X^g)_g
\]
We now show the compatibility of comparison maps with localization and Chern characters.
\begin{theorem}\label{compatibilityofloclaitions}
   Let $X$ be a complex projective algebraic variety with an algebraic action by a finite cyclic group $G=\langle g\rangle$.  The following diagram is commutative.
\begin{equation}\label{diag2}
\begin{tikzcd}[column sep=large, row sep=large]
G^{G}_0(X)
    \arrow[r, "\alpha^G"]
    \arrow[d]
&
K^{G,\mathrm{top}}_0(X)
    \arrow[r, "\lambda^G_c"]
    \arrow[d]
&
K^{G,\mathrm{an}}_0(X)
    \arrow[d]
\\
G^{G}_0(X)_g
    \arrow[r, "\alpha_g^G"]
    \arrow[d, "\overset{L_g}{\cong}"]
&
K^{G,\mathrm{top}}_0(X)_g
    \arrow[r, "(\lambda_c^G)_g"]
    \arrow[d, "\overset{\delta_g}{\cong}"]
&
K^{G,\mathrm{an}}_0(X)_g
    \arrow[d, "\overset{\gamma_g}{\cong}"]
\\
G^{G}_0(X^g)_g
    \arrow[r, "\alpha_g^G"]
      \arrow[d, "\cong"]
&
K^{G,\mathrm{top}}_0(X^g)_g
    \arrow[r, "(\lambda_c^G)_g"]
     \arrow[d, "\cong"]
&
K^{G,\mathrm{an}}_0(X^g)_g
 \arrow[d, "\cong"]
\\
G_0(X^g) \otimes R(G)_g
    \arrow[r, "\alpha \otimes 1"]
    \arrow[ddr, "\td_*\otimes\chi_g"']
&
K^{top}_0(X^g) \otimes R(G)_g
    \arrow[r, "\lambda_c \otimes 1"]
     \arrow[dd, "ch^{top}\otimes\chi_g"]
&
K^{an}_0(X^g) \otimes R(G)_g
  \arrow[ddl, "Ch\otimes\chi_g"]
  \\\\
&
H_*(X^g,\mathbb{C})
\end{tikzcd}
\end{equation}
\end{theorem}
\begin{proof} 
    Since $L_g=(i^{alg})_*^{-1}, \gamma_{{g}}=(i^{an})_*^{-1}, \delta{{g}}=(i^{top})_*^{-1}$, the commutativity of the second row  squares follows from the functoriality  of comparison maps $\alpha^G, \lambda_c^G$ with the equivariant proper inclusion $i:X^g\rightarrow X$. The  commutativity of bottom triangles comes from   the trivial  action of $G$ on $X^g$, and  the compatibility of topological, analytic Chern characters, and Todd class transformation mentioned in (\ref{compatiblitcherntodd2}).
\end{proof}
We can assume $G$ is a general finite group from the beginning by restricting to the cyclic subgroup $G'=\langle g\rangle$. The restriction of $R(G)$ to $R(G')$ induces  restriction maps
\[res^G_{G'} : K^G\rightarrow  K^{G'}\]
The naturality of comparison maps implies:
\begin{theorem}  Let $X$ be a complex projective algebraic variety with an algebraic action by a finite  group $G$. Let $g\in G$ and $G'=\langle g\rangle$.  The following diagram is commutative.
\begin{equation}\label{diag3}
\begin{tikzcd}[column sep=large, row sep=large]
G^{G}_0(X)
    \arrow[r, "\alpha^G"]
    \arrow[d, "res^G_{G'}"]
&
K^{G,\mathrm{top}}_0(X)
    \arrow[r, "\lambda^G_c"]
    \arrow[d, "res^G_{G'}"]
&
K^{G,\mathrm{an}}_0(X)
    \arrow[d, "res^G_{G'}"]
\\
G^{G'}_0(X)
    \arrow[r, "\alpha^{G'}"]
&
K^{G',\mathrm{top}}_0(X)
    \arrow[r, "\lambda_c^{G'}"]
&
K^{G',\mathrm{an}}_0(X)
\end{tikzcd}
\end{equation}
\end{theorem}
\begin{definition}
      For $M\in K_0^{G,an}(X)$ and $g\in G, G'=\langle g\rangle$, the localized class $M_g$ is the image under 
       \[K_0^{G,an}(X)\xrightarrow{res^G_{G'}} K_0^{G',an}(X)\rightarrow K_0^{G',an}(X)_{{g}}\]
       Similarly, we use the same notation for  $G_0^{G}$.
\end{definition}
\subsection{Localized characteristic classes}
Let $G$ be a finite group, and $g\in G,\:G'=\langle g\rangle$. 
\begin{definition}\label{localizedchernclass} Let $X$ be a compact $G$-CW-complex. Let $\mathcal{F}\in K_0^{G,an}(X)$. Define $$Ch(\mathcal{F},g)\in H_*(X^g,\mathbb{C})$$ to be the image of $\mathcal{F}$ under the maps 
    \[K_0^{G,an}(X)\xrightarrow{res^G_{G'}} K_0^{G',an}(X)\rightarrow K_0^{G',an}(X)_{{g}}\xrightarrow[]{\gamma_{{g}}}K_0^{G',an}(X^g)_{g}\simeq K_0^{an}(X^g) \otimes R(G')_{g}\xrightarrow{Ch_g}H_*(X^g,\mathbb{C})\]
\end{definition}
\begin{definition}\label{alglocaliationcharacclass} Let $X$ be a complex projective algebraic variety with an algebraic $G$-action.   Let $\mathcal{F}\in G^G_0(X)$. Define $$\td_{*}(\mathcal{F},g)\in H_*(X^g,\mathbb{C})$$ to be the image of $\mathcal{F}$ under the maps
    \[G_0^{G}(X)\xrightarrow{res^G_{G'}} G_0^{G'}(X)\rightarrow G_0^{G'}(X)_{{g}}\xrightarrow[]{L_{{g}}}G_0^{G'}(X^g)_{g}\simeq G_0(X^g) \otimes R(G')_{g}\xrightarrow{\td_{*}\otimes \chi_g}  H_*(X^g,\mathbb{C})\]
 We set
\[\td_{2*}(-,g):=\Psi_2\td_{*}(-,g)\]
\end{definition}
\begin{remark}
    As explained in \cite[Remark 2.4]{CMSS}, $\td_{2*}(\mathcal{F},g)$ only depends on the subgroup  $G'=\langle g\rangle$. Also, $Ch([D_X^{sign,G}],g)$, whose degree, as we will see, equals  the equivariant $g$-signature,  only depends on $G'$. See \cite[Remark 6.18]{BLP}, \cite[p.~539]{AS2}. That is why we restrict from $G$ to $G'$.
\end{remark}
\subsection{$G$-signature and  analytic equivariant $L$-classes}
Let $X$ be a $n$-dimensional compact smoothly stratified Witt space, where smoothly stratified
means stratified in the sense of Thom-Mather \cite[\S2]{ABP1}. Moscovici-Wu \cite{MW} based on work of Cheeger \cite{Ch}, and later P. Albin, E. Leichtnam, R. Mazzeo, and P. Piazza \cite{ALMP} constructed a $K$-homology class $$\left[D_X^{sign}\right]\in K_n^{an}(X)$$ Here $D_X^{sign}$ is the signature operator on the regular part of $X$, endowed with a metric quasi-isometric to a piecewise flat conic metric, such as wedge (iterated-edge) metric considered in \cite{ALMP}, and the resulting class is independent of the choice of metric. As mentioned in (\ref{ch=mathcalL}), its Chern character gives the renormalized Goresky-MacPherson $L$-class $L^{an}_*(X)$.

Now let $X$ be a compact oriented even-dimensional Thom-Mather G-Witt space where $G$ preserves the orientation, in the sense of \cite[\S2]{BLP}. In particular $K_n^{an}(X)\simeq K_0^{an}$.
\begin{theorem}
    \cite[Thm. 4.1]{BLP} There is an equivariant signature $K$-homology class $\left[D_X^{sign,G}\right]\in K^{G,an}_0(X)$ by choosing a $G$-invariant wedge metric. The class does not depend on the metric.
\end{theorem}
For $a:X\rightarrow pt$, we have $a_*:K_0^{G,an}(X)\rightarrow K_0^{G,an}(pt)=R(G)$; then
\[[Ker((D^G)^+)]-[Ker((D^G)^-)]=:ind_G([D^G])=a_*([D^G])\]
Let  $\dim_{\mathbb{R}}(X)=2m$, and for the sake of argument assume that $m$ is even. Then 
\[Sign_G(X):=IH_+-IH_-\in RO(G)\subset R(G)\]
\[Sign(g,X):=tr(g_*|IH^m_+)-tr(g_*|IH^m_-)=\chi_g(Sign_G(X))\] 
\begin{theorem}
    \cite[Prop. 5.3]{BLP} $\forall g\in G$, the equivariant signature $Sign(g,X)$ is the equivariant index of $\left[D_X^{sign,G}\right]$; 
    \[Sign(g,X)=a_*(\left[D_X^{sign,G}\right])(g)=ind_G(\left[D_X^{sign,G}\right])(g)\]
\end{theorem}
Note that the degree of analytic Chern character of the associated class of a Dirac-type operator equals the index of that operator.
\begin{theorem}  \label{homologicalclass}  \cite[Lemma 6.11]{BLP}
Let $X$ be a compact oriented even-dimensional Thom-Mather G-Witt space where $G$ preserves the orientation. Then
\[Sign(g,X)=deg~Ch(\left[D_X^{sign,G}\right],g)\]
where $Ch(-,g)$ is defined in Definition \ref{localizedchernclass}.
\end{theorem}
\begin{definition}\label{localizedLclass}
    For a finite group $G$, and $X$ a compact oriented even-dimensional Thom-Mather G-Witt space where $G$ preserves the orientation, the analytic localized $L$-class is defined by
    \[L^{an}_*(X,g):=Ch(\left[D_X^{sign,G}\right],g) \in H_*(X^g,\mathbb{C})\]
\end{definition}
\section{Relation between analytic and topological equivariant $L$-classes}\label{signvariousL}
In this section we clarify the relation between the localized class $L^{an}_*(X,g)$ and the equivariant class $L_*(X,g)$ of \cite{B26}.

Let $G$ be a finite group, and let $X$ be a compact oriented even-dimensional triangulated Witt pseudomanifold on which $G$ acts by orientation-preserving simplicial maps. For $g\in G$, the equivariant $L$-class $L_*(X,g)$ is defined by M. Banagl \cite[Def. 9.2]{B26} as follows:
\[
L_*(X,g):=\frac{1}{\deg\: \pi}\pi_!S(g)\in H_*(X,\mathbb{C})^G
\]
where $\pi:X\rightarrow X/G$ is the orbit projection, $\pi_!$ is the  \textit{transfer map} of $\pi$   $$\pi_!:H_*(X/G,\mathbb{C})\rightarrow H_*(X,\mathbb{C})$$
$S(g)\in H_*(X/G,\mathbb{C})$ comes from an equivariant version of the definition of $L$-classes; namely
\[S(g)=S_p(g):\pi^i(X)^G\otimes \mathbb{C}\rightarrow \mathbb{C}\]
\[[f]\otimes \lambda\mapsto \lambda Sign(g,f^{-1}(p))\]
for a sufficiently general $p\in S^i$ and $f:X\rightarrow S^i$. Since the equivariant signature of a closed
oriented Witt pseudomanifold $X$ is a simplicial Witt $G$-bordism invariant \cite{BLP,B26}, the definition is independent of the $G$-homotopy class of $f$ and the choice of $p$.
\begin{remark}
    As explained in \cite{B26}, $X$ can be smoothly stratified, in the sense of \cite{ABP1} by taking the open simplices as strata and $G$ acts by stratified diffeomorphisms. $X$  admits control data giving it the structure of a $G$-Thom-Mather 
space, in the sense of \cite[Def. 2.1]{BLP}. Therefore, the equivariant analytic signature class of \cite{BLP} is defined.
\end{remark}
\begin{theorem}\label{Banagl_localization_comparison}
Let $G$ be a finite group and let $X$ be a compact oriented even-dimensional triangulated Witt pseudomanifold on which $G$ acts by orientation-preserving simplicial maps. Let $g\in G$, and 
$
j:X^g\rightarrow X
$
be the inclusion of the fixed point set. Then
\[
L_*(X,g)=j_*(\Psi_2(L^{an}_*(X,g)))
\]
where $L_*(X,g)$ is Banagl's equivariant class \cite{B26} and
\[
L^{an}_*(X,g):=Ch(\left[D_X^{sign,G}\right],g)=
Ch_g\circ\gamma_g
\left(
\left[D_X^{sign,G}\right]_g
\right)
\in H_*(X^g,\mathbb{C})
\]
\end{theorem}
\begin{proof}
We have 
\[L_*(X,g)=L_0(g,X)+L_2(g,X)+\ldots\]
where $L_{i}(g,X)\in H_i(X,\mathbb{C})^G$. Assume $\dim X\leq 2(i-1)$, then  by
\cite[Prop. 9.3]{B26},
$L_i(g,X)$ is the $G$-invariant homology class
uniquely determined by the following relation (\ref{localhomologydetermined})  for all
$G$-invariant homotopy classes of maps $
f:X\rightarrow S^i
$, where $G$ acts trivially on $S^i$.  Let
$p\in S^i$ be a sufficiently general point, as explained below, $i$ is an even integer. Let
\[
u\in H^i(S^i,\mathbb{Z})
\]
be the generator compatible with the orientation. Then
\begin{equation}\label{localhomologydetermined}
\langle f^*(u),L_i(g,X)\rangle
=
Sign(g,f^{-1}(p))
\end{equation} 
 Therefore, it suffices to prove that the analogue of
(\ref{localhomologydetermined}) holds for
$
j_*(\Psi_2(L^{an}_i(X,g))).
$
This is Theorem \ref{singspheremapformula}. The dimensional restriction $\dim X\leq 2(i-1)$ can be removed by considering $X\times S^{k}$, for $k$ a sufficiently large even number, where $G$ acts
trivially on $S^{k}$, as explained in \cite{B26,Z}.
The cross  product with the fundamental class induces
\[\varphi_X:H_{i}({X},\mathbb{C}) \rightarrow H_{i+k}({X}\times S^k,\mathbb{C})\]
\[\alpha \rightarrow \alpha \otimes [S^k]\]
for $k$ large enough, $H_{i+k}(X)=0$, hence $H_{i+k}({X}\times S^k,\mathbb{C})\cong H_i(X)\otimes H_k(S^k)\cong H_i(X)$, so $\varphi_X$ is an isomorphism. One defines $L_i(g,X)$ as $\varphi_X^{-1} (L_{i+k}(g,X\times S^k))$. On the other hand, let $j\times id: X^g\times S^k\cong (X\times S^k)^g\hookrightarrow X\times S^k$ be the inclusion of fixed point set. To compare to the analytic class  we need to show
\begin{equation}\label{needtoshowcrosssphereLclass}
    \varphi_X(j_*(\Psi^{X^g}_2(L^{an}_i(X,g))))=(j\times id)_*(\Psi^{X^g\times S^k}_2(L^{an}_{i+k}(X\times S^{k},g)))
\end{equation}
We have
\begin{equation} \label{MULTIPLICATIVEX_s^K}
    [D^{sign,G}_{X\times S^k}]=[D^{sign,G\times id}_{X\times S^k}]=[D_X^{sign,G}]\boxtimes [D_{S^k}^{sign}]
\end{equation}
this follows from the multiplicativity  of signature operator; the non-equivariant case is proved in \cite[Corollary~5.18]{ABP1}, where the proof is a corollary of functoriality for fibrations \cite[Corollary 5.16]{ABP1}, where its proof is  based  on Kucerovsky’s unbounded characterization of the Kasparov product. As mentioned in the proof of \cite[Thm.~4.3]{BLP} the latter theorem holds in the equivariant settings for finite groups.

Also
$$Ch_*([D_{S^k}^{sign}])=L_*^{an}(S^k)=2^{k/2}L_k(S^k)=2^{k/2}[S^k]$$ (note that for $i<k, 0<k,~L_i(S^k)=0$  and $L_k(S^k)=[S^k]$ as positive-degree Pontryagin classes of $S^k$ vanish).
Since localization and the Chern character are multiplicative (\ref{MULTIPLICATIVEX_s^K}), this implies
\[L_{i+k}^{an}(X\times S^k,g)=L^{an}_i(X,g) \otimes Ch_k([D_{S^k}^{sign}])= L^{an}_i(X,g) \otimes 2^{k/2}[S^k] \]
Therefore
\begin{align*}
    (j\times id)_*(\Psi^{X^g\times S^k}_2(L^{an}_{i+k}(X\times S^{k},g)))=& (j\times id)_*(\Psi^{X^g\times S^k}_2( L^{an}_*(X,g) \otimes 2^{k/2}[S^k]))\\
    =& (j\times id)_*(\Psi^{X^g}_2(L^{an}_i(X,g))\otimes [S^k])\\
    =&\left(j_*(\Psi^{X^g}_2(L^{an}_i(X,g)))\right)\otimes [S^k]
\end{align*}
This proves (\ref{needtoshowcrosssphereLclass}). 
\end{proof}
By \cite[\S9]{B26} we may  choose a $G$-invariant  representative
$f:X\rightarrow S^i$ of the given $G$-invariant homotopy class with the following properties. $G$ acts trivially on $S^i$, such
that $p\in \Delta_i^{\circ}\subset S^i$ is a regular value in the  interior of an $i$-simplex, and $Y=f^{-1}(p)$
 is a $G$-invariant  triangulated Witt pseudomanifold on which $G$ acts by orientation-preserving simplicial maps, and  the inclusion $\iota:Y=f^{-1}(p)\hookrightarrow X$ is normally nonsingular with trivial normal bundle. As M. Banagl explained to the author,   $\iota_g:Y^g=Y\cap X^g\hookrightarrow X^g$ is also normally nonsingular with trivial normal bundle. Indeed, by \cite[\S9]{B26}, $f$ has the property that 
     $f^{-1}(\Delta_i^{\circ})\simeq \Delta_i^{\circ} \times Y$. $f|_{ \Delta_i^{\circ} \times Y}$ is the projection to $\Delta_i^{\circ}$ and $g$ acts trivially on the first component. Therefore, given a neighborhood $Y\hookrightarrow U\subset X$, with $U\cong \mathbb{R}^{i}\times Y$,  $U\cap X^g$ gives a neighborhood $Y^g\hookrightarrow U\cap X^g \subset X^g$ with $U\cap X^g \cong \mathbb{R}^{i}\times Y^g$.
     
To prove Theorem \ref{singspheremapformula} we need an equivariant analytic Gysin map, which we recall now. Let $\iota:Y\hookrightarrow X$ be a $G$-equivariant normally
nonsingular inclusion between even-dimensional $G$-Witt pseudomanifolds with normal bundle $N_{Y/X}$. There is an equivariant analytic Gysin map 
\begin{equation}
    \iota_{an}^{!}:
K_0^{G,an}(X)[\frac{1}{2}]
\rightarrow
K_0^{G,an}(Y)[\frac{1}{2}]
\end{equation}
\begin{equation}\label{gysin=sphires}
    \iota_{an}^{!}=s^!\circ \phi_*\circ res_U
\end{equation}
where $res_U:K_0^{G,an}(X)[\frac{1}{2}]\rightarrow K_0^{G,an}(U)[\frac{1}{2}]$ is the restriction to the tubular neighborhood $U\subset X$, $\phi_*:K_0^{G,an}(U)[\frac{1}{2}]\rightarrow K_0^{G,an}(N_{Y/X})[\frac{1}{2}]$ is induced by the tubular-neighborhood diffeomorphism, and 
$$s^!: K_0^{G,an}(N_{Y/X})[\frac{1}{2}] \rightarrow K_0^{G,an}(Y)[\frac{1}{2}]$$ 
\[\alpha \longmapsto \left(2^{\frac{-Rank_{\mathbb{R}}~ N_{Y/X}}{2}}\Sigma(\pi)\right)^{-1}\otimes \alpha\]
is the Gysin map of the zero section $s:Y\rightarrow N_{Y/X}$, where $\pi:N_{Y/X} \rightarrow Y$ is the projection, and $\Sigma(\pi)\in KK_0^G(N_{Y/X},Y)[\frac{1}{2}]$ is a well-defined element defined in \cite[Thm. 4.3]{BLP}  represented by the vertical family of signature
operators along the fibers. We normalized with $2^{\frac{-Rank_{\mathbb{R}}N_{Y/X}}{2}}$, hence the localization with respect to $2$ , to be consistent with \cite[Def. 5.11]{ABP1}. Note that the inverse means $$\left(2^{\frac{-Rank_{\mathbb{R}}N_{Y/X}}{2}}\Sigma(\pi)\right)^{-1} \otimes \left(2^{\frac{-Rank_{\mathbb{R}}N_{Y/X}}{2}}\Sigma(\pi)\right)=1\in KK_0^G(Y,Y)[\frac{1}{2}]$$
By \cite[Thm. 6.7]{ABP1} and \cite[Thm. 4.3]{BLP} we have
\begin{equation}\label{gysinequivantsignature}
    \iota_{an}^{!}([D_X^{sign,G}])
=
2^{\frac{Rank_{\mathbb{R}}~ N_{Y/X}}{2}}[D_Y^{sign,G}]
\end{equation}
\begin{theorem}\label{singspheremapformula}
Let $X, G,g,\:f:X\rightarrow S^i$ satisfy the hypotheses mentioned above, and let
$j:X^g\hookrightarrow X$ be the inclusion. Set $Y=f^{-1}(p)$. Let
$u\in H^i(S^i,\mathbb{Z})$ be the orientation generator. Then
\begin{equation}\label{localhomologydeterminedformyclass}
\left\langle
f^*(u),j_*\left(\Psi_2\left(L^{an}_*(X,g)\right)\right)
\right\rangle
=
\operatorname{Sign}(g,Y)
\end{equation}
\end{theorem}
\begin{proof}
Let $Y^g=Y\cap X^g=(f|_{X^g})^{-1}(p)$ and $\iota_g:Y^g\hookrightarrow X^g$. As explained above, $\iota_g$ is a normally nonsingular inclusion with trivial normal bundle. By assumptions $\dim X$ and $i=\operatorname{codim}_X Y$ are even, and $[D_X^{sign}]\in K^{an}_{d_X}(X)\cong K^{an}_0(X)$. By Theorem \ref{homologicalclass},
\[
Sign(g,Y)
=
\deg(L^{an}_*(Y,g))
=
\deg\: Ch_g\circ\gamma_g^Y([D_Y^{sign,G}]_g)
\]
Therefore, it suffices to prove
\[
\deg\: Ch_g\circ\gamma_g^Y([D_Y^{sign,G}]_g)
=
\langle f^*(u),j_*(\Psi_2(Ch_g\circ\gamma_g^X([D_X^{sign,G}]_g)))\rangle
\]
Since $\dim Y=\dim X-i$, by (\ref{gysinequivantsignature}) we get
\begin{equation}
        \iota_{an}^{!}([D_X^{sign,G}])
=
2^{i/2}[D_Y^{sign,G}]
\end{equation}
By Lemma  \ref{localigysincomopatibility}, we have the localization base-change identity for $\xi\in K_0^{G,an}(X)[\frac{1}{2}]$
\begin{equation}
    \gamma_g^Y((\iota_{an}^{!}\xi)_g)
=
(\iota_g)_{an}^{!}(\gamma_g^X(\xi_g))
\end{equation}
and also
\begin{equation}\label{homogysinpis2}
    (\iota_g)_H^{!}\circ\Psi_2^{X^g}
=
2^{-i/2}\Psi_2^{Y^g}\circ(\iota_g)_H^{!}
\end{equation}
where $(\iota_g)_H^{!}:H_*(X^g,\mathbb{C})\rightarrow H_{*-i}(Y^g,\mathbb{C})$ is the homological Gysin map. 
Using the above identities, Lemma  \ref{gysintopanalcomparison} applied to $\iota_g$, the triviality of its normal bundle,
we obtain
\begin{equation}\label{gysinlocallclass}
(\iota_g)_H^{!}
\bigl(
\Psi^{X^g}_2(Ch_g\circ\gamma^X_g([D_X^{sign,G}]_g))
\bigr)
=
\Psi^{Y^g}_2
\bigl(
Ch_g\circ\gamma^Y_g([D_Y^{sign,G}]_g)
\bigr)
\end{equation}
 By the projection formula,
\[
\langle f^*(u),j_*(\Psi_2(Ch_g\circ\gamma_g^X([D_X^{sign,G}]_g)))\rangle
=
\langle (f|_{X^g})^*(u),\Psi_2(Ch_g\circ\gamma_g^X([D_X^{sign,G}]_g))\rangle
\]
Let $\widetilde{u}_p\in H^i(S^i,S^i\setminus\{p\};\mathbb{Z})$ be the local orientation class at $p$ mapping to $u$. Then
\[
\tau_{\iota_g}:=(f|_{X^g})^*(\widetilde{u}_p)
\in H^i(X^g,X^g\setminus Y^g;\mathbb{Z})
\]
is the Thom class of $\iota_g$, and its image in $H^i(X^g,\mathbb{Z})$ is $(f|_{X^g})^*(u)$. By the Thom-class definition of the homological Gysin map, for every $a\in H_*(X^g,\mathbb{C})$,
\[
\langle (f|_{X^g})^*(u),a\rangle
=
\deg((\iota_g)^!_H(a))
\]
Applying this to
\[
a=\Psi_2(Ch_g\circ\gamma_g^X([D_X^{sign,G}]_g))
\]
and using (\ref{gysinlocallclass}), we obtain
\[
\langle f^*(u),j_*(\Psi_2(Ch_g\circ\gamma_g^X([D_X^{sign,G}]_g)))\rangle
=
\deg\bigl(
\Psi_2(Ch_g\circ\gamma_g^Y([D_Y^{sign,G}]_g))
\bigr)
\]
Since $\Psi_2$ acts trivially on degree-zero homology, the right-hand side equals
\[
\deg\: Ch_g\circ\gamma_g^Y([D_Y^{sign,G}]_g)
=
Sign(g,Y)
\]
\end{proof}
\begin{lemma}\label{gysintopanalcomparison}
Let $X$ and $Y$ be compact smoothly stratified spaces, and let
$\iota:Y\hookrightarrow X$ be an oriented normally nonsingular
inclusion with trivial normal bundle $N_{Y/X}$ of even rank
$r=2k$.  Then, for every
$\zeta\in K_0^{an}(X)[\frac{1}{2}]$
\begin{equation}\label{gysintopanalcomparisoneq}
\iota_H^!(Ch_X(\zeta))
=
Ch_Y(\iota_{an}^!\zeta)
\end{equation}
where $\iota_H^{!}:H_*(X,\mathbb{C})\rightarrow H_{*-r}(Y,\mathbb{C})$ is the homological Gysin map and \[
\iota_{an}^{!}:
K_0^{an}(X)[\frac{1}{2}]
\rightarrow
K_0^{an}(Y)[\frac{1}{2}]
\]
is the analytic Gysin map.
\end{lemma}
\begin{proof} Choose an oriented framing of $N_{Y/X}$ and equip it with the
induced $spin^c$ structure.
First note that in the case of trivial normal bundle  the oriented analytic Gysin map  $\iota_{an}^!$ (constructed from the signature operator of normal bundle, see \cite{ABP1} ) 
agrees with the $spin^c$ analytic Gysin map
$\iota_{an,spin^c}^!$   (constructed from the $spin^c$ Dirac operator). Indeed,  let
\[
\pi:N_{Y/X}\cong Y\times\mathbb{R}^{2k}\longrightarrow Y
\]
be the normal-bundle projection. Let
\[
\alpha_N^{sign}
=
2^{-k}[D_{N_{Y/X}/Y}^{sign}]
\in
KK_{2k}(C_0(N_{Y/X}),C(Y))[\tfrac{1}{2}]
\]
be the normalized fiberwise signature element, and let
\[
\alpha_N^{spin^c}
\in
KK_{2k}(C_0(N_{Y/X}),C(Y))
\]
be the fiberwise $spin^c$ Dirac element determined by
the framing. Denote their inverse elements by
$\beta_N^{sign}$ and
$\beta_N^{spin^c}$, respectively. For an  oriented even-dimensional $spin^c$ manifold $M$, the signature operator can be regarded as a Dirac operator on the complexified Clifford bundle and  its complex Clifford bundle is $Cl_{\mathbb{C}}(T^*M)\cong End(S_M)\cong S_M\otimes S^*_M$ \cite[Lemma~13]{RW2}, therefore
\[
[D_{\mathbb{R}^{2k}}^{sign}]
=[S^*_{\mathbb{R}^{2k}}] \cap 
[D_{\mathbb{R}^{2k}}^{spin^c}] =2^k [D_{\mathbb{R}^{2k}}^{spin^c}]
\]
where $S^*_{\mathbb{R}^{2k}}\simeq {\underline{\mathbb{C}}}^{2^k}$ is the dual spinor bundle. Therefore
$
\alpha_N^{sign}
=
\alpha_N^{spin^c}
$
and 
$
\beta_N^{sign}
=
\beta_N^{spin^c}.
$\\
Let $U\subset X$ be a tubular neighborhood of $Y$, let
$
\varphi:N_{Y/X}\longrightarrow U
$
be the tubular-neighborhood stratified diffeomorphism, and let
$
e_U:C_0(U)\longrightarrow C(X)
$
be the extension by zero. The two analytic Gysin transfer  elements are
\[
\Sigma_{sign}(\iota)
=
\beta_N^{sign}
\otimes_{C_0(N_{Y/X})}
[(\varphi^{-1})^*]
\otimes_{C_0(U)}
[e_U]
\]
\[
\Sigma_{spin^c}(\iota)
=
\beta_N^{spin^c}
\otimes_{C_0(N_{Y/X})}
[(\varphi^{-1})^*]
\otimes_{C_0(U)}
[e_U]
\]
Hence
$
\Sigma_{sign}(\iota)
=
\Sigma_{spin^c}(\iota),
$ so
 $
\iota_{an}^!(\zeta)
=
\iota_{an,spin^c}^!(\zeta).
$ Let
\[
\lambda_c:K_0^{top}(-)\longrightarrow K_0^{an}(-)
\]
be the comparison isomorphism mentioned in Section (\ref{compairsonmaps}). By \cite[Thm. 6.6]{ABP2},
$\lambda_c$ commutes with the topological and analytic
$spin^c$ Gysin maps. Thus, the following diagram
commutes:
\begin{equation}\label{diag333proof}
\begin{tikzcd}
H_*(X,\mathbb{C}) \arrow{d}{\iota_H^{!}} \arrow[<-]{r}{ch_X^{top}}
& K_0^{top}(X)
\arrow[r,"\lambda_c"]
\arrow[d,"\iota_{top}^!"']
&
K_0^{an}(X)
\arrow[d,"\iota_{an,spin^c}^!"]
\\
H_*(Y,\mathbb{C}) \arrow[<-]{r}{ch_Y^{top}} & K_0^{top}(Y)
\arrow[r,"\lambda_c"]
&
K_0^{an}(Y).
\end{tikzcd}
\end{equation}
 See \cite[\S6]{ABP2} for the construction of $\iota_{top}^!$. The left square commutes as the normal  bundle is trivial. Let $\eta=\lambda_c^{-1}(\zeta)$. 
By the commutativity of \eqref{diag333proof} and the equality of the
two analytic Gysin maps,
\[
\iota_{an}^!(\zeta)
=
\lambda_c(\iota_{top}^!\eta)
\]
Since $Ch=ch^{top}\circ\lambda_c^{-1}$, see (\ref{lambda-cabp2}),  we obtain $Ch_Y(\iota_{an}^!\zeta)
=
ch_Y^{top}(\iota_{top}^!\eta)$.  Therefore,
\[
\iota_H^!(Ch_X(\zeta))=\iota_H^!(ch_X^{top}(\eta))
=
ch_Y^{top}(\iota_{top}^!\eta)=Ch_Y(\iota_{an}^!\zeta)
\]
\end{proof}
\begin{lemma}\label{localigysincomopatibility}
Let $G=\langle g\rangle$ be a finite cyclic group. Let $X$ and $Y$ be compact even-dimensional $G$-spaces (in the sense of \cite{BLP}) having the $G$-homotopy type of
finite $G$-$CW$ complexes. Suppose that
$\iota:Y\hookrightarrow X$ is a $G$-equivariant normally
nonsingular inclusion preserving orientation. Assume that the induced inclusion
$\iota_g:Y^g\hookrightarrow X^g$ is  normally nonsingular
and that
$
N_{Y^g/X^g}\cong N_{Y/X}|_{Y^g}
$
as $G$-equivariant oriented vector bundles, compatibly with the normally nonsingular tubular data. Then for every
$\xi\in K_0^{G,an}(X)[\frac{1}{2}]$,
\[
\gamma_g^Y((\iota_{an}^{!}\xi)_g)
=
(\iota_g)_{an}^{!}(\gamma_g^X(\xi_g))
\]
\end{lemma}
\begin{proof}
Let $i_{X^g}:X^g\hookrightarrow X$ and $i_{Y^g}:Y^g\hookrightarrow Y$ be the
fixed-point inclusions. Consider
\begin{equation}\label{diaglocalgysin}
\begin{tikzcd}
K_0^{G,an}(X^g)_g[\frac{1}{2}]
\arrow[r,"(\iota_g)_{an}^{!}"]
\arrow[d,"(i_{X^g})_*"']
&
K_0^{G,an}(Y^g)_g[\frac{1}{2}]
\arrow[d,"(i_{Y^g})_*"]
\\
K_0^{G,an}(X)_g[\frac{1}{2}]
\arrow[r,"\iota_{an}^{!}"']
&
K_0^{G,an}(Y)_g[\frac{1}{2}]
\end{tikzcd}
\end{equation}
  The maps
$(i_{X^g})_*$ and $(i_{Y^g})_*$ are induced by the Kasparov
product with $[i_{X^g}^*]$ and $[i_{Y^g}^*]$, the  restriction homomorphisms of $C^*$-algebras, with inverses
$\gamma_g^X$ and $\gamma_g^Y$. Therefore, it suffices to prove  the commutativity of (\ref{diaglocalgysin}). By (\ref{gysin=sphires}) 
$$\iota_{an}^{!}=s^!\circ \phi_*\circ res_U$$
where $$s^!: K^{G,an}(N_{Y/X})[\frac{1}{2}] \rightarrow K^{G,an}(Y)[\frac{1}{2}]$$ is the Gysin map of zero-section $s:Y\rightarrow N_{Y/X}$. Similarly, define these symbols with $(-)_g$ for the map $(\iota_g)_{an}^{!}$. By Kasparov multiplicativity,  to show the commutativity of (\ref{diaglocalgysin}), the only nontrivial part is to show the commutativity of  (as $\phi_*$ and $res_U$ are just a diffeomorphism and a restriction)
\begin{equation}\label{diaglocalgysin11}
\begin{tikzcd}
K_0^{G,an}(N_{Y^g/X^g})_g[\frac{1}{2}]
\arrow[r,"s_g^{!}"]
\arrow[d,"(r_{g})_*"']
&
K_0^{G,an}(Y^g)_g[\frac{1}{2}]
\arrow[d,"(i_{Y^g})_*"]
\\
 K^{G,an}(N_{Y/X})_g[\frac{1}{2}]
\arrow[r,"s^{!}"']
&
K_0^{G,an}(Y)_g[\frac{1}{2}]
\end{tikzcd}
\end{equation}
where $(r_{g})_*:K_0^{G,an}(N_{Y^g/X^g})_g \rightarrow K^{G,an}(N_{Y/X})_g$ is the Kasparov product with the restriction homomorphism $C_0(N_{Y/X})\rightarrow C_0(N_{Y^g/X^g})$. This is the base change for Gysin restriction to zero-sections applied to the following Cartesian diagram, proved (non-equivariantly) in \cite[Prop. 6.10]{ABP1}. The equivariant version follows identically. 
\begin{equation}
\begin{tikzcd}
N_{Y^g/X^g}
\arrow{r}{r_g}
\arrow{d}{\pi_g}
&
N_{Y/X}
\arrow{d}{\pi}
\\
Y^g
\arrow{r}{i_{Y^g}}
&
Y
\arrow[ul,phantom,very near start,"\lrcorner"]
\end{tikzcd}
\end{equation}
\end{proof}
\subsection{$L$-classes of the orbit space}
As explained in the introduction, by \cite{B26}, one can express the Goresky--MacPherson $L$-class of $X/G$ as an average of equivariant $L$-classes.  The following corollary follows immediately from Theorem  \ref{banaglintro} and Theorem \ref{Banagl_localization_comparison}.
\begin{corollary}\label{averageLclass}
   Let $G$ be a finite group and $X$ an oriented closed Witt pseudomanifold upon
which $G$ acts simplicially, preserving the orientation.  We have
    \begin{equation}\label{averageLclasseq}
        L^{an}_*(X/G)=\frac{1}{|G|}\sum_{g\in G} \pi^g_*Ch(\left[D_X^{sign,G}\right],g)=\frac{1}{|G|}\sum_{g\in G} \pi^g_*(L^{an}_*(X,g))
    \end{equation}
\end{corollary}
\begin{remark} 
    One can give another proof of (\ref{averageLclasseq}) after applying  the analytic Lefschetz--Riemann--Roch, Theorem \ref{LRRan}, if one proves
\begin{equation}\label{siangureclassofqutientktheory}
(\pi_*\left[D_X^{sign,G}\right])^G=[D_{X/G}^{sign}]  \in K^{an}_0(X/G)
\end{equation}
    As suggested to the author by E. Leichtnam, this may be possible by expressing $L^2$ forms on $X/G$ by the equivariant $L^2$ forms on $X$. Note that (\ref{siangureclassofqutientktheory}) corresponds under Conjecture \ref{G-K-CONJECTURE} to the fact 
        \[MHC_1(IC_{X/G}'^H)\cong MHC^G_1((\pi_* IC_{X}'^H)^G)=(\pi_*(MHC^G_1(IC_{X}'^H)))^G\]
    where the first equality comes from $IC_{X/G}'^H\cong (\pi_* IC_{X}'^H)^G$ proved in
    \cite[Lemma 5.6]{CMSS}.

\end{remark}
\section{Explicit formula for localization}
In this section we review explicit formulas for the computation of localization maps $L_g$ and $\gamma_g$. These formulas show that, under suitable transversality and normal nonsingularity conditions on the inclusion $i:X^g\hookrightarrow X$, a localized class can be expressed as a pullback together with a (co)normal factor.

\subsection{Analytic localized class $\gamma_g([D^{sign,G}_X]_{{g}})$}
In this section, we review a result of \cite{BLP} on an explicit formula for $L^{an}_*(X,g)$  under a strongly normally nonsingular inclusion assumption on $X^g\hookrightarrow X$.

Let $G=\langle g\rangle$. Let $X$ be a compact even-dimensional oriented G-Witt pseudomanifold. For simplicity of notation we assume $X^g$ is connected, hence $F=X^g$ in the language of \cite{BLP}. Assume $i:X^g\rightarrow X$ is $G$-equivariantly strongly normally nonsingular, in the sense of \cite[Def. 2.10]{BLP}. we recall the definition.
\begin{definition} \label{BLP_normallynonsing}\cite[Def. 2.10]{BLP} Let $i:F \hookrightarrow X$ be an equivariant inclusion of smoothly Thom--Mather $G$-stratified spaces.  It is \textit{$G$-equivariantly strongly normally nonsingular} if there exists a $G$-equivariant real vector bundle $p : E \rightarrow F$, a $G$-invariant open neighborhood $U \subset X$ of $F$, and a $G$-equivariant stratified diffeomorphism $\phi:U \rightarrow E$ respecting the relevant families of equivariant control data, and whose restriction to $F$ is the zero section
inclusion $F \subset E$.
\end{definition}
There is a  stratified diffeomorphism $U_g\rightarrow E_g$ with $E_g$  the $G$-equivariant normal bundle of $X^g$ and $U_g$ a $G$-equivariant tubular neighborhood of $X^g$ in $X$. There is a decomposition of the normal bundle by the eigenvalues of the action of $g$:
\[E_g= E_g(-1) \bigoplus_{0<\theta_j<\pi} E_c(e^{i\theta_j}) \] 
$E_g(-1)$ is an oriented even-dimensional real vector bundle on which $g$ acts by $-1$, and $ E_c(e^{i\theta_j})$ admits a $G$-invariant complex structure where $g$ acts by $e^{i\theta_j}$. Assume $E_g(-1)$ admits a complex structure, and the normal bundle $E_g$ is a $G$-equivariant complex vector bundle, hence oriented.  The orientation of $X$ and $E_g$ induces an orientation on $X^g$.
\begin{theorem}\cite[Prop. 6.13]{BLP} 
    With the above assumptions,
\begin{equation}\label{localizedfrmulablp613}
    \gamma_{{g}}(\left[D_X^{sign,G}\right]_{{g}})= (-1)^{\dim_{\mathbb{C}}E_g}[[\mathcal{E}_{X^g}]]\otimes \left[D_{X^g}^{sign}\right]_{{g}}
\end{equation}
where
\[[[\mathcal{E}_{X^g}]] \in KK_0^G(X^g,X^g)_{{g}}, \: \left[D^{sign, G}_{X^g}\right]_{{g}}\in K^{G,an}_0(X^g)_{{g}}=KK_0^G(X^g,\mathbb{C})_{{g}}\]
where $[[\mathcal{E}_{X^g}]]$ is given by the product 
\[\bigotimes[[\mathcal{E}(e^{i\theta_j})]]\otimes[[\mathcal{E}(e^{i\pi})]]\]
where the following classes are given by "quotient" virtual bundles which have meaning in the localized
K-theory group  in $KK^G_0(X^g,X^g)_g\cong KK_0(X^g,X^g)\otimes R(G)_g$
\[[[\mathcal{E}(e^{i\theta_j})]]=\frac{[[\bigwedge E_c(e^{i\theta_j})]]}{[[\bigwedge^{even} E_c(e^{i\theta_j})]]-[[\bigwedge^{odd} E_c(e^{i\theta_j})]]}\]
When $\theta=\pi,\: $
\[[[\mathcal{E}(-1)]]=\frac{[[S(E_g(-1))]]}{[[S^+(E_g(-1))]]-[[S^-(E_g(-1))]]}\]
\end{theorem}
\begin{remark}\label{comatbileoreination}
    Note that in \cite[Thm. 6.13]{BLP} and \cite[Prop.~3.8]{R}   there is no extra sign $(-1)^{\dim_{\mathbb{C}}E_g}$ as they use a different orientation of normal bundle. In this paper we follow the orientations as in  \cite{AS3} ,  there is a compatibility of orientations $TX|_{X^g}=TX^g \oplus E_g $. As explained in  \cite[4]{R}, the contribution of each normal subbundle differs from \cite{AS3} by a sign, i.e., $\textbf{Ch}_g([[\mathcal{E}(e^{i\theta})]])=(-1)^{\dim_{\mathbb{C}}N_{\theta}} C^{\theta}_1$, as seen in (\ref{normalcalcualtion}). We verify our correction sign in Example \ref{examplesign}.
\end{remark}
There is a bivariant Chern character \cite{P}, which for separable  $C^*$-algebras $A,B$ 
\[\textbf{Ch}:KK_0(A,B)\rightarrow HC^{*}_{loc}(A,B)\]
the  subscript $loc$ means local cyclic cohomology. In the case of $A=B=C(X^g)$ this gives \cite{BLP}
\[\textbf{Ch}:KK_0(X^g,X^g)\rightarrow Hom^{even}(H^{even}(X^g),H^{even}(X^g))\]
It's also compatible with the analytic Chern character $Ch:KK_0(X^g,\mathbb{C})\rightarrow H_*(X^g,\mathbb{C})$ by taking $A=C(X^g), B=\mathbb{C}$. As usual the subscript $(g)$ denotes composition with the evaluation map $\chi_g$. Puschnigg proved that this bivariant Chern character transforms Kasparov product into composition
\begin{equation}\label{Chgclass}
 \textbf{Ch}_g([[\mathcal{E}_{X^g}]]\otimes [D_{X^g}^{sign,G}]_{{g}})=\textbf{Ch}_g([[\mathcal{E}_{X^g}]])\circ Ch(\left[D_{X^g}^{sign}\right])
\end{equation}
where the homological class $Ch(\left[D_{X^g}^{sign}\right])$ is viewed as a form $H^*\rightarrow \mathbb{C}$. Also $Ch(\left[D_{X^g}^{sign}\right])=L^{an}_*(X^g)$. The class (\ref{Chgclass}) is a linear form on $H^{*}$ and it's given by \cite[Eq. $(23)$]{BLP}
\[v\mapsto \langle v\wedge Ch_g(\mathcal{E}_{X^g}),L^{an}_*(X^g)\rangle\]
Hence as a homology class it is identified with
\[Ch_g(\mathcal{E}_{X^g})\cap L^{an}_*(X^g)\]
where
\begin{align}
      Ch_g(\mathcal{E}(e^{i\theta_j}))=\frac{ch_g\:\bigwedge E_c(e^{i\theta_j})}{ch_g \bigwedge^{even} E_c(e^{i\theta_j})-ch_g \bigwedge^{odd} E_c(e^{i\theta_j})}=&\prod_l\frac{1+e^{x^l_j}e^{i\theta_j}}{1-e^{x^l_j}e^{i\theta_j}}\\
      =&\label{normalcalcualtion}(-1)^{\dim_{\mathbb{C}}E_c(e^{i\theta_j})}\prod_l  \coth(\frac{x^l_j+i\theta_j}{2})
\end{align}

 $\{x^l_j\}_l$ are Chern roots of $E_c(e^{i\theta_j})$. This also applies to $\theta_0=\pi$, assuming $E_g(-1)$ admits a $G$-invariant complex structure (which is the case if the setup comes from complex  geometry). Note that using the complex structure on $E_c(e^{i\theta_j})$ we can rewrite
\[Ch_g(\mathcal{E}(e^{i\theta_j}))=\frac{{ch}_g(\lambda_{1}E_c(e^{i\theta_j}))}{{ch}_g(\lambda_{-1}E_c(e^{i\theta_j}))}\]
by viewing $E_c(e^{i\theta_j})\in K_{G}^0(X^g)$ and $ch_g:K_{G}^0(X^g)\rightarrow H^*(X^g,
\mathbb{C})$.
\begin{corollary}\label{chernlocal} Assume
$i:X^g\rightarrow X$ is $G$-equivariantly strongly normally nonsingular.
If the normal bundle $E_g$ is even-dimensional and $E_g(-1)$ admits a $G$-invariant complex structure, set $\theta_0=\pi$. Then
    \[L^{an}_*(X,g):=Ch(\left[D_X^{sign,G}\right],g)=Ch_g\circ  \gamma_g(\left[D^{sign,G}_X\right]_{{g}})=(-1)^{\dim_{\mathbb{C}}E_g}  \prod_j {Ch}_g(\mathcal{E}(e^{i\theta_j})) \cap L^{an}_*(X^g)= \]\[(-1)^{\dim_{\mathbb{C}}E_g}\prod_j  \frac{{ch}_g(\lambda_1(E_c(e^{i\theta_j})))}{{ch}_g(\lambda_{-1}(E_c(e^{i\theta_j})))}   \cap L^{an}_*(X^g) \in H_*(X^g,\mathbb{C})\]
\end{corollary}

\begin{example}\label{examplesign}
    Let $G=\{1,g\}=\mathbb{Z}_2$ act on $\mathbb{CP}^2$, by $g[x:y:z]=(-x,y,z)$. Then the fixed point set is $X^g= \mathbb{CP}^1 \cup [1:0:0]$. The normal bundle of $\mathbb{CP}^1$ is $E:=E_{\mathbb{CP}^1}(-1)=\mathcal{O}_{\mathbb{CP}^1}(1)$. $H^2(\mathbb{CP}^2)$ is generated by the hyperplane class which is Poincar\'e dual to $[\mathbb{CP}^1]$, hence $g$ acts trivially on  $H^2(\mathbb{CP}^2)$, so $Sign(g,\mathbb{CP}^2)=1$. The usual complex structure on $\mathcal{O}_{\mathbb{CP}^1}(1)$ is $G$-invariant and  complex orientations of $T\mathbb{CP}^2|_{\mathbb{CP}^1}, T\mathbb{CP}^1, \mathcal{O}_{\mathbb{CP}^1}(1)$ are all compatible. We have $L^{an}_*(\mathbb{CP}^1)=2[\mathbb{CP}^1]$. The isolated fixed point $[1:0:0]$ doesn't contribute to the $g$-signature as the Chern roots of normal bundle vanish. We have
    \[\langle c_1(\mathcal{O}(1)),[\mathbb{CP}^1 ]\rangle=\deg(\mathcal{O}(1))=1\]
    Then $x= c_1(\mathcal{O}(1)), x^2=0 \rightarrow e^x=1+x$ 
    \[Ch_g(\mathcal{E}_{\mathbb{CP}^1}(-1))=\frac{1-e^x}{1+e^x}=\frac{-x}{2+x}=\frac{-x}{2}\]
    \[(-1)^{dim_{\mathbb{C}}E}\langle Ch_g(\mathcal{E}_{\mathbb{CP}^1}(-1)),L^{an}_*(\mathbb{CP}^1) \rangle =(-1) \langle \frac{-x}{2}, 2[\mathbb{CP}^1] \rangle=1=Sign(g,\mathbb{CP}^2)\]
\end{example}

\subsection{Algebraic localization formula via embedding  }\label{algebraiclocalizationembeding}
In this section, we review an explicit localization formula in the algebraic setting. Let
$G=\langle g\rangle$ be a finite cyclic group acting algebraically on a projective
complex algebraic variety $X$. 

As mentioned in Theorem \ref{coherencetraceformula1}, there is an isomorphism \cite{Q79,T1,Ta}
\[
L_g:G_0^G(X)_g\xrightarrow{\cong}G_0^G(X^g)_g
\]
which is the inverse of  the isomorphism
induced by the inclusion $i:X^g\hookrightarrow X$:
\[
i^{alg}_*:G_0^G(X^g)_g\xrightarrow{\cong}G_0^G(X)_g
\]
When $X$ is smooth and the inclusion $i:X^g\hookrightarrow X$ is a
regular embedding, we may identify
\[
G_0^G(X)\cong K_G^0(X)
\]
where $ K_G^0(X)$ is the Grothendieck group of $G$-equivariant locally free
$\mathcal{O}_X$-modules. In this case, the localization map is given by, for $[\mathcal{F}]\in K_G^0(X)$
\[
L_g([\mathcal{F}]_g)
=
i^*[\mathcal{F}]_g
\otimes
\lambda_{-1}(N^\vee)_g^{-1}
\]
where $N^\vee=N_{X^g/X}^\vee$ is the conormal bundle of $i$, and
$
i^*:K_G^0(X)_g\longrightarrow K_G^0(X^g)_g
$
is the equivariant pullback.

In the singular case, in order to describe the localized class, we choose
a $G$-equivariant closed immersion $j:X\hookrightarrow Z$ into a smooth proper $G$-variety $Z$. Since we are assuming $X$ to be projective such a $Z$ exists (for instance, by taking the product of $|G|$ copies of a projective
space containing $X$, with $G$ acting by permutation of the factors). Then $j$ induces a closed immersion of fixed point sets $j_g:X^g\hookrightarrow Z^g$. Using the identification of $G_0^G(X)$ with the equivariant $G_0$-theory of
algebraic coherent sheaves on $Z$ supported on $X$, we obtain the following formula.
\begin{theorem}(\cite{BFQ}, \cite[p.~75]{Ta}) \label{coherencetraceformula} Let $X$ be a projective complex algebraic variety, with an algebraic action by a finite cyclic group $G=\langle g\rangle$. Then for $[\mathcal{F}]\in G_0^G(X)$, the localized class
$$
L_g([\mathcal{F}]_g)=\mathcal{G}_g\in G_0^G(X^g)_g
$$
is represented, under the identification of $G_0^G(X^g)$ with the
equivariant $G_0$-theory of $Z^g$ supported on $X^g$, by the virtual coherent $G$-sheaf 
\begin{equation}\label{localform}
\mathcal{G}_g
= \lambda_{-1}
\left(
N_{Z^g/Z}^{\vee}
\right)_g^{-1}
\otimes_{\mathcal{O}_{Z^g}}
\left(
j_*\mathcal{F}
\otimes_{\mathcal{O}_Z}^{L}
\mathcal{O}_{Z^g}
\right)_g
\end{equation}
Equivalently,
\[
\mathcal{G}_g
=\left(
\sum_q(-1)^q
\left[
\operatorname{Tor}^{\mathcal{O}_Z}_q
\left(
\mathcal{O}_{Z^g},\mathcal{O}_{Z^g}
\right)
\right]
\right)_g^{-1}
\otimes_{\mathcal{O}_{Z^g}}
\left(
\sum_q(-1)^q
\left[
\operatorname{Tor}^{\mathcal{O}_Z}_q
\left(
j_*\mathcal{F},\mathcal{O}_{Z^g}
\right)
\right]
\right)_g
\]
\end{theorem}
We still need to compute  $j_*\mathcal{F}\otimes^L \mathcal O_{Z^g}$. To be able to do this we will impose some transversality conditions on the embedding $X^g\hookrightarrow X$. 

Suppose there exists a $G$-equivariant
    closed immersion
    $
    j:X\hookrightarrow Z
    $ into a smooth proper $G$-variety $Z$ such that
after identifying $X$ with its image, $X^g=X\cap Z^g$,
     $X$ and $Z^g$ are Tor-independent in $Z$. Equivalently, the
    following $G$-equivariant Cartesian square is Tor-independent:
    \begin{equation}\label{diag8}
    \begin{tikzcd}
    X^g \arrow{d}{i_X} \arrow{r}{j_g}
    & Z^g \arrow{d}{i_Z} \\
    X \arrow{r}{j}
    & Z
    \end{tikzcd}
    \end{equation}
    that is
    \[
    \operatorname{Tor}^{\mathcal{O}_Z}_k
    (\mathcal{O}_X,\mathcal{O}_{Z^g})=0
    \qquad\text{for all }k>0.
    \]
Since the Cartesian square (\ref{diag8}) is Tor-independent,
derived base change yields a natural isomorphism
$
Li_Z^*Rj_*
\xrightarrow{\cong}
R(j_g)_*Li_X^*.
$
Since $j$ and $j_g$ are closed immersions, their pushforward functors
are exact, and hence
$
Li_Z^*j_*
\xrightarrow{\cong}
(j_g)_*Li_X^*.
$ Passing to $G_0$-theory gives
\[
i_Z^!\circ j_*=(j_g)_*\circ i_X^!:
G^G_0(X)\longrightarrow G^G_0(Z^g)
\]
Indeed, for a regular embedding $i:Y\hookrightarrow X$, the Gysin map
$$
i^!:G^G_0(X)\longrightarrow G^G_0(Y)
$$
is defined by
\[
i^![\mathcal{F}]
=
\sum_q(-1)^q
\left[
\operatorname{Tor}^{\mathcal{O}_X}_q
(\mathcal{O}_Y,\mathcal{F})
\right]
\]
for $\mathcal{F}\in\operatorname{Coh}^G(X)$, hence there is an
isomorphism
\[
i_Z^!\bigl(j_*[\mathcal{F}]\bigr)=[j_*\mathcal{F}
\otimes_{\mathcal{O}_Z}^{L}
\mathcal{O}_{Z^g}]\cong
[(j_g)_*(\mathcal{F}\otimes_{\mathcal{O}_X}^{L}\mathcal{O}_{X^g})]=(j_g)_*\bigl(i_X^![\mathcal{F}]\bigr)
\]
in $G^G_0(Z^g)$.
\begin{corollary}\label{algebriaclocalfomrulasimplfied}
    Let $i:X^g\rightarrow X$. Assume there exists a $G$-equivariant
    closed immersion
    $
    j:X\hookrightarrow Z$ satisfying the above condition, then
     \[L_g:G_0^G(X)_g\rightarrow G_0^G(X^g)_g\]
     \begin{equation}\label{localformmm}
         L_g([\mathcal{F}]_g)=  (\lambda_{-1}N^{\vee}_{X^g/X})_g^{-1}\otimes_{\mathcal{O}_{X^g}}\bigl(i^![\mathcal{F}]\bigr)_g 
     \end{equation}
\end{corollary}
\begin{proof} Let $i_X=i$ as in (\ref{diag8}).
By Theorem \ref{coherencetraceformula}
\[
L_g([\mathcal{F}]_g)
=
\lambda_{-1}
\left(
N_{Z^g/Z}^{\vee}
\right)_g^{-1}
\otimes_{\mathcal{O}_{Z^g}}
\left(
j_*\mathcal{F}
\otimes_{\mathcal{O}_Z}^{L}
\mathcal{O}_{Z^g}
\right)_g
=
\lambda_{-1}
\left(
N_{Z^g/Z}^{\vee}
\right)_g^{-1}
\otimes_{\mathcal{O}_{Z^g}}
(j_g)_*\bigl(i_X^![\mathcal{F}]\bigr)_g 
\]
 The Tor-independence condition implies $(j_g)^*N^{\vee}_{Z^g/ Z}=N^{\vee}_{X^g/X}$. By the projection formula
\[
(\lambda_{-1}N^{\vee}_{Z^g/Z})_g^{-1}
\otimes_{\mathcal{O}_{Z^g}}
(j_g)_*\bigl(i_X^![\mathcal{F}]\bigr)_g
=
(j_g)_*
\left(
(j_g)^*(\lambda_{-1}N^{\vee}_{Z^g/Z})_g^{-1}
\otimes_{\mathcal{O}_{X^g}}
\bigl(i_X^![\mathcal{F}]\bigr)_g
\right)
=
\]
\[
\cong
(j_g)_*
\left(
(\lambda_{-1}N^{\vee}_{X^g/X})_g^{-1}
\otimes_{\mathcal{O}_{X^g}}
\bigl(i_X^![\mathcal{F}]\bigr)_g
\right)
\cong
(\lambda_{-1}N^{\vee}_{X^g/X})_g^{-1}
\otimes_{\mathcal{O}_{X^g}}
\bigl(i_X^![\mathcal{F}]\bigr)_g 
\]
    under the identification of $G_0^G(X^g)$ with elements of $G_0^G(Z^g)$ with support on $X^g$. 
\end{proof}
If we plug $\mathcal{F}=MHC^G_1(IC'^H_X)$ into (\ref{localformmm}), we need to compute  $i_X^!(MHC^G_1(IC'^H_X))$. To be able to do this we will impose a further normal nonsingularity condition on the inclusion of fixed point set.
\subsection{The Gysin class $i^!(MHC^G_1(IC'^H_X))$}
Let $i:X^g\hookrightarrow X$. The goal of this section is to compute  $i^!(MHC^G_1(IC'^H_X))$ under the Gysin map
\[i^!:G^G_0(X)\longrightarrow G^G_0(X^g)\]
by imposing a normal nonsingularity condition.

The notion of an \textit{upwardly normally nonsingular} closed regular
algebraic embedding of complex algebraic varieties was introduced by
M.~Banagl. If $V\hookrightarrow U$ is a closed regular embedding of complex varieties, $\beta:Bl_VU\rightarrow U$ will denote the blow-up of $U$ along $V$. The exceptional divisor $E = \beta^{-1}V$  is the
projectivization of the algebraic normal bundle of $V$ in $U$.

\begin{definition}\cite[Defs. 6.1 and 6.5]{B}\label{upward}
A closed regular algebraic embedding $Y\hookrightarrow X$ is
\textit{upwardly normally nonsingular} if:
\begin{enumerate}
    \item Its underlying topological embedding is normally nonsingular
    and compatibly stratifiable, and its topological normal bundle is
    isomorphic to the underlying topological vector bundle of the
    algebraic normal bundle $N_{Y/X}\rightarrow Y$.
    \item The inclusion $E\subset \operatorname{Bl}_{Y\times 0}(X\times\mathbb{C})$, of the exceptional divisor $E$ of the blowup $X\times\mathbb{C}$ along  $Y\times 0$,  is topologically normally
    nonsingular.
\end{enumerate}
\end{definition}
\begin{theorem} \label{pullbackofmhc1} Let $G=\langle g\rangle$ be a finite cyclic group acting on pure-dimensional projective complex algebraic varieties $X,Y$ by
algebraic automorphisms.
If  $i:Y\hookrightarrow X$ is a $G$-equivariant upwardly normally nonsingular embedding with the algebraic (co)normal bundle $N(N^{\vee})$, then
\[
i^! MHC^G_1(IC'^H_X)=\lambda_1(N^\vee)\otimes MHC^G_1(IC'^H_{Y})
\]
Here $i^!:G^G_0(X)\rightarrow G^G_0(Y)$ is the Gysin map and $K_G^0(Y)\otimes G^G_0(Y)\rightarrow G^G_0(Y)$ is the $K_G^0(Y)$-module structure of $G^G_0(Y)$. $\lambda_1(N^\vee)$ denotes the equivariant $\lambda$-class of the $G$-equivariant conormal bundle.
\end{theorem}
The non-equivariant homological version was proved in \cite[Thm. 6.30]{B};
\[
i^! \td_{2*}(MHC_1(IC'^H_X))=L^*(N)\cap \td_{2*}(MHC_1(IC'^H_{Y}))
\]
We follow the argument given there using deformation to normal cone except applying a $G$-equivariant $G_0$-theoretic  description of the Gysin map in terms of  Verdier specialization map instead of a homological one \cite[Lemma 6.29]{B}.
\begin{proof}[Proof of Theorem \ref{pullbackofmhc1}]
Recall from Appendix, Theorem \ref{equiavarintspecialzition}, that
 the Gysin map $i^!:G^G_0(X)\rightarrow G^G_0(Y)$ factors as
\begin{equation}\label{diag4}
\begin{tikzcd}
G^G_0(X)  \arrow{rr}{i^!} \arrow{rd}{Sp_G} && G^G_0(Y)   \\
&G^G_0(N) \arrow{ru}{k^!} &
\end{tikzcd}
\end{equation}
where $k: Y  \hookrightarrow N $ is the inclusion of the zero section, and $Sp_G$ is the $G$-equivariant $G_0$-theoretic Verdier specialization map. Let
\[
p:\mathcal{Z}^{\circ}\longrightarrow \mathbb{C}
\]
be the (equivariant) deformation to the normal bundle (with trivial action on $\mathbb{C}$) associated to the embedding
$i:Y\hookrightarrow X$. Thus,
\[
p^{-1}(0)=N,
\qquad
p^{-1}(\mathbb{C}^{*})\cong X\times\mathbb{C}^{*}
\]
Denote the corresponding inclusions and projection by
\[
\iota:N\hookrightarrow\mathcal{Z}^{\circ},
\qquad
j:X\times\mathbb{C}^{*}\hookrightarrow\mathcal{Z}^{\circ},
\qquad
q:X\times\mathbb{C}^{*}\longrightarrow X
\]
We review the proof of \cite[Thm. 6.30]{B}. We show
\begin{equation}\label{spmhc}
Sp_G (MHC^G_1(IC'^H_X))=MHC^G_1(IC'^H_N)
\end{equation}
\begin{equation}\label{k!mhc}
k^!(MHC^G_y(IC'^H_N))=\lambda_y(N^{\vee})\otimes MHC^G_y(IC'^H_{Y})
\end{equation}
To show $(\ref{k!mhc})$, let $\pi: N \rightarrow Y$ be the projection and
\[
\pi_K^*:G^G_0(Y)\rightarrow G^G_0(N), \qquad \pi^*:K_G^0(Y)\rightarrow K_G^0(N), \qquad k^!:G^G_0(N)\rightarrow G^G_0(Y)
\]
Then by equivariant Sch\"urmann's Verdier-Riemann-Roch, Theorem \ref{equiv_vrr_shurmann}
\[
MHC^G_{y}(\pi_{\mathrm{mhm}}^*([M]))=\lambda_y(T^*_{\pi})\otimes \pi_K^*MHC^G_{y}([M])
\]
with $T^*_{\pi}$ the equivariant relative cotangent bundle. Using $\pi_{\mathrm{mhm}}^*(IC'^H_Y)=IC'^H_N$ \cite[Lemma 6.13]{B} we get
\[
MHC^G_{y}(IC'^H_N)=\lambda_y(T^*_{\pi})\otimes \pi_K^*MHC^G_{y}(IC'^H_Y)
\]
then $T^*_{\pi}=\pi^*(N^\vee)$
\[
k^!(MHC^G_{y}(IC'^H_N))=k^!(\lambda_y(\pi^*(N^\vee))\otimes \pi_K^*MHC^G_{y}(IC'^H_Y))
\]
\[
=k^!(\pi^*\lambda_y(N^\vee)\otimes \pi_K^*MHC^G_{y}(IC'^H_Y))=k^!\pi_K^*(\lambda_y(N^\vee)\otimes MHC^G_{y}(IC'^H_Y))=\lambda_y(N^\vee)\otimes MHC^G_{y}(IC'^H_Y)
\]
since $k^!\circ \pi_K^*=id_{G^G_0(Y)}$.

We prove a slightly stronger version of (\ref{spmhc}), namely
\begin{equation}\label{spmhcy}
Sp_G\bigl(MHC^G_y(IC'^H_X)\bigr)=MHC^G_y(IC'^H_N)
\end{equation}
Since restriction to an open subset commutes with $MHC_y$, we have
\[
j^*MHC^G_y(IC_{\mathcal{Z}^\circ}^{\prime H})
=
MHC^G_y(IC_{X\times\mathbb{C}^*}^{\prime H})
\]
\[
q_{\mathrm{mhm}}^*(IC_X^{\prime H})
=
IC_{X\times\mathbb{C}^*}^{\prime H}
\]
Again,  by equivariant Sch\"urmann's Verdier-Riemann-Roch, Theorem \ref{equiv_vrr_shurmann}, for the smooth morphism $q$, we get
\[
MHC^G_y(IC_{X\times\mathbb{C}^*}^{\prime H})
=
\lambda_y(T_q^*)\otimes q^*MHC^G_y(IC_X^{\prime H})
\]
The equivariant relative cotangent bundle $T_q^*$ is the trivial line bundle with trivial $G$-representation, hence
\[
\lambda_y(T_q^*)=1+y
\]
Consequently,
\[
j^*MHC^G_y(IC_{\mathcal{Z}^\circ}^{\prime H})
=
(1+y)q^*MHC^G_y(IC_X^{\prime H})
\]
By the definition of $Sp_G$, (\ref{sp_g_description}), this implies
\begin{equation}\label{spg1+y}
(1+y)Sp_G\bigl(MHC^G_y(IC_X^{\prime H})\bigr)
=
\iota^!MHC^G_y(IC_{\mathcal{Z}^\circ}^{\prime H})
\end{equation}
 Equivariant Sch\"urmann's specialization theorem, Theorem \ref{equiv_shurman_specialziation}, gives
\[
\iota^!MHC^G_y(IC_{\mathcal{Z}^\circ}^{\prime H})
=
(1+y)MHC^G_y\bigl(\Psi_p^{\prime H,G}(IC_{\mathcal{Z}^\circ}^{\prime H})\bigr)
\]
where
\[
\Psi_p^{\prime H,G}:=\Psi_p^{H,G}[1]
\]
is the shifted Hodge nearby-cycle functor associated to $p$. Let $n=\dim_{\mathbb{C}}X$. Then
\[
IC_{\mathcal{Z}^\circ}^{\prime H}
=
IC_{\mathcal{Z}^\circ}^H[-n-1]
\]
Since $i$ is G-equivariant upwardly normally  nonsingular, by an equivariant version of  
\cite[Prop. 6.20]{B} mentioned in Theorem \ref{banagl_ic_nomralcone_z} we have
\[
\Psi_p^{H,G}(IC_{\mathcal{Z}^\circ}^H)\cong IC_N^H
\]

It follows that
\[
\Psi_p^{\prime H,G}(IC_{\mathcal{Z}^\circ}^{\prime H})
=
\Psi_p^{H,G}[1](IC_{\mathcal{Z}^\circ}^H[-n-1])
=
\Psi_p^{H,G}(IC_{\mathcal{Z}^\circ}^H)[-n]
=
IC_N^H[-n]
=
IC_N^{\prime H}
\]
Thus,
\[
\iota^!MHC^G_y(IC_{\mathcal{Z}^\circ}^{\prime H})
=
(1+y)MHC^G_y(IC_N^{\prime H})
\]
Combining this with (\ref{spg1+y}), we obtain
\[
(1+y)Sp_G\bigl(MHC^G_y(IC_X^{\prime H})\bigr)
=
(1+y)MHC^G_y(IC_N^{\prime H})
\]
Multiplication by $1+y$ is injective on
$G^G_0(N)[y,y^{-1}]$, so
\[
Sp_G\bigl(MHC^G_y(IC_X^{\prime H})\bigr)
=
MHC^G_y(IC_N^{\prime H})
\]
\end{proof}
\subsection{\textit{G-nicely} normally nonsingular and algebraic localized class $L_g(MHC^{G}_1(IC'^H_X)_g)$}
In this section, we collect all normally nonsingular type conditions needed on the inclusion of a fixed point set $X^g \hookrightarrow X$ in order to compute and compare the  analytic and algebraic localized  classes. 
\begin{definition}\label{nice}
Let $X$ be a projective complex algebraic variety, and let
$G=\langle g\rangle$ be a finite cyclic group acting on $X$ by
algebraic automorphisms. We say that the inclusion of the fixed-point
set $X^g\hookrightarrow X$ is  \textit{G-nicely} normally nonsingular if:
\begin{enumerate}
    \item The inclusion $i:X^g\hookrightarrow X$ is an upwardly normally
    nonsingular embedding, Definition \ref{upward}.
    \item The inclusion $i:X^g\hookrightarrow X$ is $G$-equivariantly
    strongly normally nonsingular, Definition \ref{BLP_normallynonsing}.
    \item There exists a smooth proper $G$-variety $Z$ and a $G$-equivariant
    closed immersion
    \[
    j:X\hookrightarrow Z
    \]
    such that, after identifying $X$ with its image, $X^g=X\cap Z^g$,
    and $X$ and $Z^g$ are Tor-independent in $Z$. Equivalently, the
    following $G$-equivariant Cartesian square is Tor-independent:
    \begin{equation}\label{diag9}
    \begin{tikzcd}
    X^g \arrow{d}{i_X} \arrow{r}{j_g}
    & Z^g \arrow{d}{i_Z} \\
    X \arrow{r}{j}
    & Z
    \end{tikzcd}
    \end{equation}
    that is
    \[
    \operatorname{Tor}^{\mathcal{O}_Z}_k
    (\mathcal{O}_X,\mathcal{O}_{Z^g})=0
    \qquad\text{for all }k>0.
    \]
\end{enumerate}
\end{definition}

\begin{example}
Assume that there exists a smooth proper $G$-variety $Z$, together with a $G$-equivariant closed
immersion $j:X\hookrightarrow Z$, such that $X$ and $Z^g$ are
Tor-independent in $Z$. Assume also that $X$ admits a $G$-invariant
algebraic Whitney stratification whose strata are transverse to $Z^g$
in $Z$. Let $$\beta:\operatorname{Bl}_{Z^g\times 0}(Z\times\mathbb{C})\rightarrow Z\times\mathbb{C}$$ be the blowup. Assume that the strict transform $\beta^{-1}(X\times \mathbb{C})$
admits a possibly non-algebraic Whitney stratification transverse to
the exceptional divisor $\mathbb{P}(N_{Z^g\times 0/Z\times \mathbb{C}})\hookrightarrow \operatorname{Bl}_{Z^g\times 0}(Z\times\mathbb{C})$. Then, the embeddings
$
X\hookrightarrow Z\hookleftarrow Z^g
$
are upwardly transverse in the sense of \cite[Def. 6.4]{B}.
Consequently, $X^g\hookrightarrow X$ is upwardly normally nonsingular
by \cite[Corollary 6.7]{B}.  Moreover, $X^g\hookrightarrow X$ is $G$-equivariantly strongly normally
nonsingular by \cite[Thm. 3.1]{BLP}. Therefore,
$X^g\hookrightarrow X$ is  \textit{G-nicely} normally nonsingular.
\end{example}
\begin{example}
    If $X=Z$ is smooth, $X^g\hookrightarrow X$ is a submanifold and its normal bundle is a $G$-equivariant vector bundle. The blowup $\operatorname{Bl}_{Z^g\times 0}(Z\times\mathbb{C})$
    is smooth and the exceptional divisor is a smooth divisor smoothly embedded. As explained in \cite[Ex.~6.2]{B} the smooth tubular neighborhood theorem  provides topological normal nonsingularity. Moreover by \cite[Thm.~3.1]{BLP} it's $G$-strongly normally nonsingular,  thus, it is  \textit{G-nicely} normally nonsingular. 
\end{example}
We will compute the algebraic localized Hirzebruch class $\td_{2*}(MHC^G_{1}(IC'^H_X),g)$ under the  assumption that the inclusion of fixed point set $i:X^g\hookrightarrow X$ is \textit{G-nicely} normally nonsingular. Note that the action of $G=\langle g\rangle$ on the fixed point set $X^g$ is trivial, and $IC'^H_{X^g}$ is given its trivial equivariant structure. Corollary  \ref{algebriaclocalfomrulasimplfied} and Theorem \ref{pullbackofmhc1} imply
\begin{corollary}\label{nice_alg_localized_formula}
Let $G=\langle g\rangle$ be a finite cyclic group acting  on  a  pure-dimensional complex projective variety $X$ by
algebraic automorphisms. If  the inclusion of  fixed point set $i_X:X^g\hookrightarrow X$ is  \textit{G-nicely} normally nonsingular, then
\[
\begin{aligned}
L_g(MHC^{G}_1(IC'^H_X)_g) &= (\lambda_{-1}N^{\vee}_{X^g/ X})_g^{-1} \otimes i^!_X MHC^{G}_1(IC'^H_X)_g \\
&= (\lambda_1N^{\vee}_{X^g/ X})_g\otimes(\lambda_{-1}N^{\vee}_{X^g/X})_g^{-1} \otimes MHC^{G}_1(IC'^H_{X^g})_g
\end{aligned}
\]
   \end{corollary}
\begin{corollary}\label{tdgformula} Let  $G$ be a finite group acting on  a pure-dimensional  complex projective  algebraic variety  $X$ by
algebraic automorphisms. Let $G'=\langle g\rangle$.  Assume the inclusion of  fixed point set $X^g\subset X$  is  \textit{$G'$-nicely} normally nonsingular (Definition \ref{nice}). Then
   \[
\begin{aligned}
\td_{2*}(MHC^{G}_1(IC'^H_X),g) &= (\td_{2*}\otimes \chi_g)(L_{{g}}(MHC^{G}_1(IC'^H_X)_g)) \\
&= (\td_{2*}\otimes \chi_g)\left((\lambda_1N^{\vee}_{X^g/ X})_g\otimes(\lambda_{-1}N^{\vee}_{X^g/X})_g^{-1} \otimes MHC^{G}_1(IC'^H_{X^g})_g\right) \\
&= (ch_{(2),g}(\lambda_1(N^{\vee}_{X^g/ X})_g)\cup ch_{(2),g}(\lambda_{-1}N^{\vee}_{X^g/ X})_g^{-1}) \cap (\td_{2*}\otimes \chi_g) MHC^{G}_1(IC'^H_{X^g})_g \\
&= (ch_{(2),g}(\lambda_1N^{\vee}_{X^g/X})_g\cup ch_{(2),g}(\lambda_{-1}N^{\vee}_{X^g/ X})_g^{-1}) \cap IT_{1*}(X^g)
\end{aligned}
\]
\end{corollary}
\section{Orbit space $X/G$}\label{orbitspace}
Let $G$ be a finite group. Note that if $G$ acts algebraically on a quasi-projective complex algebraic variety,  then $X/G$ exists and is also a quasi-projective complex algebraic variety. If $X$ is an oriented closed Witt pseudomanifold upon which $G$ acts simplicially, preserving the orientation, then $X/G$ is a Witt space \cite{B26}.  The theme of this section is expressing various algebraic and analytic characteristic classes of the quotient space $X/G$ as an average of localized classes on the fixed point sets. The idea behind this is a linear algebra fact \cite[p.~21]{HZ}: If $\rho:G \rightarrow GL(V)$ is a representation of a finite group $G$, then
\begin{equation}\label{linearalgebra}
    \dim(V^G)=\frac{1}{|G|}\sum_{g\in G} \chi(g)
\end{equation}
where $V^G$ denotes the subspace of invariant vectors and $\chi(g)=\operatorname{tr}(\rho(g))$.

\subsection{Analytic Lefschetz-Riemann-Roch }
Theorem \ref{averagetoddalgebraic} which says
\[IT_{1*}(X/G)=\frac{1}{|G|}\sum_{g\in G} \pi_*^gIT_{1*}(X,g)\]
where $\pi^g:X^g\xrightarrow{i} X\xrightarrow{\pi}X/G$, follows from applying the following Lefschetz-Riemann-Roch theorem for Todd class transformation to  $MHC^{G}_1(IC'^H_X) \in G_0^G(X)$.
\begin{theorem}\cite[formula $(90)$]{BFQ,M,CMSS}\label{algerabic_LRR_todd} Let $G$ be a finite group acting by algebraic automorphisms on the complex
projective variety $X$. Let $\mathcal{F}\in G_0^G(X)=K_0(Coh^G(X))$, 
     \begin{equation}
     \td_{*}((\pi_*\mathcal{F})^G)=\frac{1}{|G|}\sum_{g\in G} \pi_*^g(\td_{*}(\mathcal{F},g))
 \end{equation}
 where $(-)^G$ denotes the projector onto the $G$-invariant sub-object.
\end{theorem}
It is natural to ask if the analogue holds in $K^{G,an}_0(X)$.
\begin{theorem}\label{LRRan}
Let $G$ be a finite group, and let $X$ be a compact $G$-space having the $G$-homotopy type of a finite $G$-$CW$ complex.  For any $M\in K_0^{G,an}(X)$,
\begin{equation}
    Ch((\pi_*M)^G)
=
\frac{1}{|G|}
\sum_{g\in G}
\pi_*^g
\bigl(
Ch(M,g)
\bigr)
\end{equation}
where $\pi^g:X^g\xrightarrow{i} X\xrightarrow{\pi}X/G$.
\end{theorem}
\begin{proof} Let $Y:=X/G$.
Recall 
\[Ch(M,g)=Ch_g\circ\gamma_g(M_g)\]
Since $G$ acts trivially on $Y$, there is a natural identification
$
K_0^{G,an}(Y)
\cong
K_0^{an}(Y)\otimes R(G)
$.
For $A=\alpha\otimes\rho\in K_0^{an}(Y)\otimes R(G)$, we have
\[
A^G
=
\alpha\dim(\rho^G)
=
\frac{1}{|G|}
\sum_{g\in G}\alpha\chi_g(\rho)
\]
Applying the Chern character gives
\[
Ch(A^G)
=
\frac{1}{|G|}
\sum_{g\in G}
(Ch\otimes\chi_g)(A)
\]
By additivity, this identity holds for every $A\in K_0^{G,an}(Y)$.  Let $G'=\langle g\rangle$. Since $\chi_g$ factors through the localization $R(G')_g$, the map $Ch\otimes\chi_g$ factors through $K_0^{G',an}(Y)_g$. Applying the preceding identity to $A=\pi_*M$, we obtain
\[
Ch((\pi_*M)^G)
=
\frac{1}{|G|}
\sum_{g\in G}
(Ch\otimes\chi_g)((\pi_*M)_g)
\]
It remains to identify $(\pi_*M)_g$ with $\pi_*^g(\gamma_g^X(M_g))$. Restriction of the group action commutes with proper pushforward; $res^G_{G'}(\pi_*M)
=
\pi_*(res^G_{G'}M)$, therefore, $(\pi_*M)_g
=
\pi_*(M_g)
\in K_0^{G',an}(Y)_g$. Let $i_g:X^g\hookrightarrow X$ be the inclusion. By the localization theorem,
$
\gamma_g^X$
is the inverse of $(i_g)_*$. Since $G'$ acts trivially on $Y$, we have $Y^g=Y$,  $\gamma_g^Y=id$, and $\pi_*\circ(i_g)_*
=
\pi_*^g,$. The diagram
\begin{equation}\label{diag5}
\begin{tikzcd}
K_0^{G',an}(X)_g
\arrow[r,"\gamma_g^X"]
\arrow[d,"\pi_*"]
&
K_0^{G',an}(X^g)_g
\arrow[d,"\pi_*^g"]
\\
K_0^{G',an}(Y)_g
\arrow[r,"\gamma_g^Y"]
&
K_0^{G',an}(Y)_g
\end{tikzcd}
\end{equation}
commutes. Consequently,
$
(\pi_*M)_g
=
\pi_*^g(\gamma_g^X(M_g)).
$
\end{proof}
\subsection{Identification of analytic and algebraic localized classes}
So far we have obtained two homological classes in $H_*(X^g,\mathbb{C})$; an analytic localized class $Ch(\left[D_X^{sign,G}\right],g)$ and an algebraic localized  class $\td_{2*}(MHC^G_1(IC'^H_X),g)$. If we assume  Conjecture \ref{G-K-CONJECTURE} to hold for the $G$-variety $X$, then by the  compatibility of localizations we automatically get the identification of localized classes.
    \begin{theorem}\label{identificationGKcond}
     Let $G$ be a finite group acting by algebraic automorphisms on a   pure-dimensional complex projective variety $X$. Assume the equivariant $K$-theoretical Hodge index conjecture \ref{G-K-CONJECTURE} holds for $X$, that is 
\begin{equation}\label{ghkconjeceq}
    (\lambda_c^G \circ  \alpha^G)([MHC^G_{1}(IC'^H_X)])=\left[D^{sign,G}_X\right]
\end{equation}
     then 
        \begin{equation}
    IT_{1*}(X,g):=\td_{2*}(MHC^G_1(IC'^H_X),g)=\Psi_2Ch(\left[D_X^{sign,G}\right],g)=:\Psi_2L^{an}_*(X,g)
    \end{equation}
\end{theorem}
\begin{proof}
   This follows from the definitions of $\td_{2*}(-,g)$, $Ch(-,g)$  and   Theorem \ref{compatibilityofloclaitions} applied to (\ref{ghkconjeceq}).
\end{proof}
\begin{remark}
  In the case of smooth $X$, this comparison is compatible with  $T^*_{1}(X,g)=L(X,g)$ of \cite[Remark~2.3]{CMSS}.
\end{remark}
Corollary \ref{averageLclass}, Theorem \ref{identificationGKcond}, and Theorem \ref{averagetoddalgebraic} imply
\begin{corollary}
        Let $G$ be a finite group acting by algebraic automorphisms on a   pure-dimensional complex projective variety $X$. Assume the equivariant $K$-theoretical Hodge index conjecture \ref{G-K-CONJECTURE} holds for $X$. Then Conjecture \ref{BSYi} holds for the quotient space $X/G$.
\end{corollary}
If we only  assume the characteristic  Conjecture \ref{BSYi}, we need the inclusion $X^g\hookrightarrow X$ to be  \textit{$G'$-nicely} normally nonsingular, $G'=\langle g \rangle$.
\begin{theorem}\label{Identification}
  Let  $G$ be a finite group acting on  a pure-dimensional complex projective algebraic
variety $X$ by algebraic automorphisms.  Let  $g\in G, G'=\langle g \rangle$.  Assume the inclusion of fixed point set $X^g\hookrightarrow X$ is  \textit{$G'$-nicely} normally nonsingular (Definition \ref{nice}). Moreover, assume Conjecture \ref{BSYi} holds for $X^g$; $IT_{1*}(X^g)=L_*(X^g)$. Then in $H_*(X^g,\mathbb{C})$ we have
    \begin{equation}
        IT_{1*}(X,g):=\td_{2*}(MHC^G_1(IC'^H_X),g)=\Psi_2Ch(\left[D_X^{sign,G}\right],g)=:\Psi_2L^{an}_*(X,g)
    \end{equation}
\end{theorem}
\begin{proof}
We assume the orientation on $X$ is the orientation induced from the complex algebraic structure. Since the normal bundle  $N$ of $X^g\hookrightarrow X$ is a holomorphic bundle and $g$ acts holomorphically, 
\begin{equation}\label{holdecomN}
    N=\bigoplus_{0<\theta<2\pi} N_{\theta}
\end{equation}
where $g$ acts by $e^{i\theta}$ on $N_{\theta}$.  In the analytic language we need to pass to the real normal bundle,  its decomposition into real $g$-invariant  blocks is
\begin{equation}\label{realdecomN}
    E=N_{\mathbb{R}}=E_g(-1) \bigoplus_{0<\xi<\pi} E_c(e^{i\xi})
\end{equation}
where $g$ acts by $e^{i\xi}$ on $E_c(e^{i\xi})$. $E_c(e^{i\xi})$ (including $E_c(e^{i\pi}):=E_g(-1)$) is even-dimensional and admits a natural $G$-invariant complex structure, we call this the \textit{analytic} complex structure. For the $(-1)$-eigenspace we use the holomorphic complex structure inherited from $N_\pi$ and 
\[
E_c(e^{i\pi}):=E_c(-1)=N_\pi
\] For ${0<\xi<\pi}$, the relation between the analytic complex structure and the holomorphic complex structure is
\[E_c(e^{i\xi})=N_{\xi} \oplus \overline{N_{2\pi-\xi}}\]
Let $r_\xi:=\operatorname{rank}_{\mathbb C}N_\xi$, then
\[\frac{ch_g(\lambda_1N_\xi)}
     {ch_g(\lambda_{-1}N_\xi)}
=
(-1)^{r_\xi}
\frac{ch_g(\lambda_1N_\xi^\vee)}
     {ch_g(\lambda_{-1}N_\xi^\vee)}\]
\[\frac{ch_g(\lambda_1\overline{N_{2\pi-\xi}})}
     {ch_g(\lambda_{-1}\overline{N_{2\pi-\xi}})}
=
\frac{ch_g(\lambda_1N_{2\pi-\xi}^\vee)}
     {ch_g(\lambda_{-1}N_{2\pi-\xi}^\vee)}\]
Hence
\[\frac{ch_g(\lambda_1E_c(e^{i\xi}))}
     {ch_g(\lambda_{-1}E_c(e^{i\xi}))}
=
(-1)^{r_\xi}
\frac{ch_g(\lambda_1N_\xi^\vee)}
     {ch_g(\lambda_{-1}N_\xi^\vee)}
\frac{ch_g(\lambda_1N_{2\pi-\xi}^\vee)}
     {ch_g(\lambda_{-1}N_{2\pi-\xi}^\vee)}\]
Since we equip $E_g(-1)$ with the holomorphic complex structure of $N_\pi$,
\begin{equation}\label{minus-one-factor}
\frac{ch_g(\lambda_1E_c(e^{i\pi}))}
     {ch_g(\lambda_{-1}E_c(e^{i\pi}))}
=
(-1)^{r_\pi}
\frac{ch_g(\lambda_1N_\pi^\vee)}
     {ch_g(\lambda_{-1}N_\pi^\vee)}
\end{equation}
Set \[
r_+:=\sum_{0<\xi<\pi}\operatorname{rk}_{\mathbb C}N_\xi,
\qquad
r_-:=\sum_{\pi<\theta<2\pi}\operatorname{rk}_{\mathbb C}N_\theta
\]
Multiplying  over $0<\xi \leq \pi$, we obtain
\begin{equation}\label{total-normal-factor}
\prod_{0<\xi\leq\pi}
\frac{ch_g(\lambda_1E_c(e^{i\xi}))}
     {ch_g(\lambda_{-1}E_c(e^{i\xi}))}
=
(-1)^{r_++r_\pi}
\frac{ch_g(\lambda_1N^\vee)}
     {ch_g(\lambda_{-1}N^\vee)}
\end{equation}
The analytic complex structure agrees with the holomorphic complex structure on the summands $N_\xi$, $0<\xi\leq\pi$, whereas on each $N_\theta$ with $\pi<\theta<2\pi$ it is the conjugate complex structure. Replacing the complex structure of a complex vector bundle of rank $r$ by its conjugate changes the orientation of its underlying real bundle by $(-1)^r$. Hence the analytic orientation of $E$ differs from its holomorphic orientation by
$(-1)^{r_-}$. As mentioned in Remark \ref{comatbileoreination} the orientation on $X^g$ is induced by that of $X$ and $E$; 
\[TX|_{X^g}\simeq TX^g \oplus E\]
 The orientation of $E$ comes from the orientations of  analytic complex structure on $E_c(e^{i\xi})$.  $X^g$ with  this induced orientation  is denoted by $X_{an}^g$. Since the orientation on $X$ is fixed, it follows that the induced analytic orientation and the holomorphic orientation of $X^g$ also differ by $(-1)^{r_-}$. Consequently,
\begin{equation*}
L_*(X^g)
=
(-1)^{r_-}L_*(X_{\mathrm{an}}^g)
=
(-1)^{r_-}\Psi_2L^{an}_*(X_{\mathrm{an}}^g)
\end{equation*}
By assumption,
\[
IT_{1*}(X^g)=L_*(X^g)
\]
\begin{equation*}
\Psi_2L^{an}_*(X_{\mathrm{an}}^g)
=
(-1)^{r_-}IT_{1*}(X^g)
\end{equation*}
By Corollary \ref{tdgformula},
\begin{equation*}
IT_{1*}(X,g)
=
\frac{ch_{(2),g}(\lambda_1N^\vee)}
     {ch_{(2),g}(\lambda_{-1}N^\vee)}
\cap IT_{1*}(X^g)
\end{equation*}
By Corollary \ref{chernlocal},
\[
L^{an}_*(X,g)
=
(-1)^{\dim_{\mathbb C}N}
\prod_{0<\xi\leq\pi}
\frac{ch_g(\lambda_1E_c(e^{i\xi}))}
     {ch_g(\lambda_{-1}E_c(e^{i\xi}))}
\cap L^{an}_*(X_{\mathrm{an}}^g)
\]
Therefore by (\ref{total-normal-factor}) and $r_++r_-+r_\pi=\dim_{\mathbb C}N$ we get
\[
\begin{aligned}
\Psi_2L^{an}_*(X,g) &= (-1)^{\dim_{\mathbb C}N+r_++r_\pi} \frac{ch_{(2),g}(\lambda_1N^\vee)} {ch_{(2),g}(\lambda_{-1}N^\vee)} \cap \Psi_2L^{an}_*(X_{\mathrm{an}}^g) \\
&= (-1)^{\dim_{\mathbb C}N+r_++r_\pi+r_-} \frac{ch_{(2),g}(\lambda_1N^\vee)} {ch_{(2),g}(\lambda_{-1}N^\vee)} \cap IT_{1*}(X^g) \\
&= \frac{ch_{(2),g}(\lambda_1N^\vee)} {ch_{(2),g}(\lambda_{-1}N^\vee)} \cap IT_{1*}(X^g) \\
&= IT_{1*}(X,g)
\end{aligned}
\]
\end{proof}
Corollary \ref{averageLclass}, Theorem \ref{Identification}, and Theorem \ref{averagetoddalgebraic} imply
\begin{corollary}\label{theoremorbitspace}
 Let $G$ be a finite group acting by algebraic automorphisms on a   pure-dimensional complex projective variety $X$. For every $g\in G$ assume the following: For $G'=\langle g\rangle$, the inclusion of the fixed point set $X^g\hookrightarrow X$ is  \textit{G'-nicely} normally nonsingular (Definition \ref{nice}), and  Conjecture \ref{BSYi} holds for $X^g$. Then
    \[IT_{1*}(X/G)=L_*(X/G)\] 
\end{corollary}

\section{Appendix: Equivariant specialization theorems}\label{appendix}
In this appendix, we make some remarks about how some well-known specialization theorems lift to the $G$-equivariant case for a finite group $G$, such as Sch\"urmann's theorem on specialization of motivic Hodge-Chern classes \cite[Thm. 1.1]{S2}, Sch\"urmann's Verdier--Riemann--Roch formula for smooth pullbacks \cite[Corollary~5.11]{S}, Verdier specialization map, and specialization of intersection Hodge module of normal cone deformation space \cite[Prop. 6.20]{B}.  

We briefly review some equivariant notions discussed in \cite[\S~6]{CMSS}. Let $A$ be a covariant pseudofunctor on the category of quasi projective varieties, taking
values in additive categories, e.g., $A(X)$ being  $D^{b}(MHM(X)), D_{coh}^{b}(X),$ or $D_{c}^{b}(X,\mathbb{Q})$. Then a $G$-equivariant object $M\in A(X)$ is given by the underlying object $M$ together with isomorphisms
$
\psi_g:M\rightarrow g_*M
$
satisfying the cocycle condition, and an equivariant morphism is a morphism that commutes with all $\psi_g$. Let $A^G(X)$ denote the category of equivariant objects and morphisms. An equivariant morphism $f:X\rightarrow Y$ induces a pushforward $f_*^G$ and an equivariant pullback $f_G^*$, the latter always exists in the case of $A(X)=D^{b}(MHM(X)),D_{c}^{b}(X,\mathbb{Q})$ and for smooth or flat morphisms in the case of $D_{coh}^{b}(X)$. Equivariant exterior products also exist for all these categories. In the case of a triangulated category, an equivariant distinguished triangle in $A^G(X)$ is given by a distinguished triangle in $A(X)$ where all the morphisms are $G$-equivariant. Then one can define the Grothendieck group $K_0(A^G(X))$ as the quotient of the free abelian group generated by objects of $A^G(X)$ by the relations given by equivariant distinguished triangles. Lastly, the content of \cite[Ex.~6.6]{CMSS} is that given an exact functor $F:A(X)\rightarrow A'(X)$ of triangulated categories commuting with the $G$-action, there is an induced $G$-equivariant functor $F^G:A^G(X)\rightarrow A'^G(X)$ with underlying functor $F=For(F^G)$ and a group homomorphism
$$
F^G:K_0(A^G(X))\rightarrow K_0(A'^G(X))
$$
In particular, this applies to $Gr_p^FDR^G$. Note that the forgetful functor
\begin{equation}\label{forgetful}
For:A^G(X)\rightarrow A(X)
\end{equation}
is conservative; if $M\xrightarrow{f}N$ is an equivariant morphism such that $For(f)$ is an isomorphism in $A(X)$, then $f$ is an isomorphism in $A^G(X)$ (the inverse becomes equivariant).

Moreover, from \cite[Lemma.~6.7]{CMSS} it follows, e.g., that
$$
K_0(D^{b,G}(MHM(X)))\cong K_0(MHM^G(X))
$$ 
$$
K_0(D_{coh}^{b,G}(X))\cong K_0(Coh^G(X))=:G^G_0(X)
$$
where the right sides are, respectively, the $K_0$ group of the abelian category of $G$-equivariant algebraic mixed Hodge modules and coherent algebraic sheaves on $X$.

 \begin{theorem} \label{equiavarintspecialzition} 
Let $i:Y\hookrightarrow X$ be a $G$-equivariant closed regular embedding of complex algebraic varieties. Let $N$ denote the algebraic normal bundle. The Gysin map $i^!:G^G_0(X)\rightarrow G^G_0(Y)$ factors as
\begin{equation}\label{diag4}
\begin{tikzcd}
G^G_0(X)  \arrow{rr}{i^!} \arrow{rd}{Sp_G} && G^G_0(Y)   \\
&G^G_0(N) \arrow{ru}{k^!} &
\end{tikzcd}
\end{equation}
where $k: Y  \hookrightarrow N $ is the inclusion of the zero section, and $Sp_G$ is the $G$-equivariant $G_0$-theoretic Verdier specialization map
\[
Sp_G: G^G_0(X)\rightarrow G^G_0(N)
\]
\end{theorem}
\begin{proof}
    This is given non-equivariantly on the level of $G_0$ for algebraic stacks in \cite[\S~2.2]{Q}.  See \cite[Lemma~6.29]{B} for these maps on the level of $H_*$. This can be lifted to the equivariant case  for algebraic varieties, e.g., by considering the quotient stacks $[X/G]$ and $[Y/G]$ and using the natural identification
$$
G_0([X/G])\cong G_0^G(X)
$$
A $G$-equivariant embedding $i:Y\hookrightarrow X$ induces a morphism
$
i:[Y/G]\hookrightarrow [X/G]
$. Let us recall how $Sp_G$ is defined via   deformation  to the normal cone \cite[\S~2.1]{Q}. Let
\[
p:\mathcal{Z}^{\circ}\longrightarrow \mathbb{C}
\]
be the (equivariant) flat deformation to the normal bundle (with trivial action on the parameter space $\mathbb{C}$) associated to the embedding
$i:Y\hookrightarrow X$. Thus,
\[
p^{-1}(0)=N,
\qquad
p^{-1}(\mathbb{C}^{*})\cong X\times\mathbb{C}^{*}
\]
Denote the corresponding inclusions and projection by
\[
\iota:N\hookrightarrow\mathcal{Z}^{\circ},
\qquad
j:X\times\mathbb{C}^{*}\hookrightarrow\mathcal{Z}^{\circ},
\qquad
q:X\times\mathbb{C}^{*}\longrightarrow X
\]
These maps are equivariant. If $\alpha\in G^G_0(X)$  then
\begin{equation}\label{sp_g_description}
    Sp_G(\alpha)=\iota^!\circ (j^*)^{-1}\circ q^*(\alpha)
\end{equation}
where $(j^*)^{-1}$ denotes choosing a lift of a class from $G^G_0(X\times\mathbb{C^*})$ to $G^G_0(\mathcal{Z}^{\circ})$, the lift is independent of choices because of the following exact sequence
\[G_0^G(N)\xrightarrow{\iota_*} G_0^G(\mathcal{Z}^{\circ})\xrightarrow{j^*}G_0^G(X\times\mathbb{C^*})\rightarrow 0\]
if $\beta_1,\beta_2$ are two such lifts, then $\beta_1-\beta_2=\iota_*(\alpha)$ and then
$\iota^!(\beta_1-\beta_2)=\iota^!\iota_*(\alpha)=\lambda_{-1}(N^*_{\iota})\alpha$.
The normal bundle $N_{\iota}$ of central fiber is the trivial line bundle, and hence  $\lambda_{-1}(N^*_{\iota})=0$.
\end{proof}
Recall
\[
 MHC^G_y: K_0(\MHM^G(X)) \to  G^G_0(X)[y,y^{-1}]
\]
is the $G$-equivariant motivic Chern class transformation, defined in \cite{CMSS}, from the Grothendieck group of $G$-equivariant algebraic mixed Hodge modules to the Grothendieck group of $G$-equivariant algebraic coherent sheaves on $X$. As mentioned in the introduction,  the construction of $MHC^G_y$ in \cite{CMSS} relies on the notion of \textit{weak equivariant derived categories} \cite{CMSS}  to adapt Saito’s
functors $\gr^F_{-p}\DR$ to the equivariant context.
 \begin{theorem}(Equivariant VRR for smooth pullbacks)\label{equiv_vrr_shurmann}
    Let $G$ be a finite group acting by algebraic automorphisms on quasi-projective complex algebraic varieties $Z,Z'$. Let  $q:Z'\rightarrow Z$ be a  $G$-equivariant smooth morphism. Then for $M\in \mhm^G(Z)$
\[
MHC^G_y(q^*[M])
=
\lambda_y(T_q^*)\otimes q^*MHC^G_y([M])
\]
where  $T^*_{q}$ is the equivariant relative cotangent bundle of $q$. $\otimes$ denotes the $K^0$-module structure of $G_0$.
\end{theorem}
\begin{proof}
The non-equivariant proof given in \cite[Corollary 5.11]{S} uses smooth base change, logarithmic de Rham description, and the projection formula. By \cite[\S~6]{CMSS}, proper pushforward, exterior products, smooth pullback, and projection formula lift to the equivariant context. The equivariant version of logarithmic de Rham description \cite[Ex.~5.8]{S} is already proved in \cite[Eq.~(62)]{CMSS}.
\end{proof}

Now  let $G$ be a finite group acting by algebraic automorphisms on a quasi-projective complex algebraic variety $X$. Let $f:X\rightarrow \mathbb{C}$ be an equivariant algebraic function with the trivial action on $\mathbb{C}$. Let $X_0:=\{f=0\}$ be a hypersurface of codimension one. Since $f$ is $G$-equivariant, the nearby cycle functor $\Psi^H_f$ on  algebraic mixed Hodge modules, which is an exact functor, commutes with the push-forward $g_*$ induced by each $g \in G$, and hence by \cite[Ex.~6.6]{CMSS} this induces a $G$-equivariant functor on the weakly derived category
\[\Psi^{H,G}_f:D^{b,G}\mhm(X)\rightarrow D^{b,G}\mhm(X_0)\]
and a group homomorphism 
\[\Psi^{H,G}_f:K_0(D^{b,G}\mhm(X))\rightarrow K_0(D^{b,G}\mhm(X_0))\]
then by the isomorphism   $K_0(D^{b,G}\mhm(X))\cong K_0(\mhm^G(X))$ \cite[Lemma~6.7]{CMSS} this induces a group homomorphism
\[\Psi^{H,G}_f:K_0(\mhm^G(X))\rightarrow K_0(\mhm^G(X_0))\]
Let $i:X_0\rightarrow X$ be the inclusion and
\[i^!:G^G_0(X)\rightarrow G^G_0(X_0)\]
is the Gysin map mentioned in Theorem \ref{equiavarintspecialzition}. Let
\[\Psi'^{H,G}_f:=\Psi^{H,G}_f[1]\]
 \begin{theorem}(Equivariant specialization of motivic Hodge-Chern classes)\label{equiv_shurman_specialziation}
We have the equality
\[(1+y)MHC^G_y(\Psi'^{H,G}_f(-))=i^!(MHC^G_y(-))\]
 as transformations $K_0(\mhm^G(X)) \rightarrow G_0^G(X_0)[y,y^{-1}]$.
 \end{theorem}
\begin{proof} We follow \cite{S2,MS}. \(MHC_y^G\), equivariant nearby cycles, and the equivariant
Gysin morphism commute with proper \(G\)-pushforward, and equivariant
resolution of singularities reduces the general case to the following case of algebraic manifold $M$. The proof uses only the underlying filtered $\sD$-module $(\mathcal M,F)$. Let the graph embedding be
$
i_f:M\hookrightarrow M':=M\times\mathbb C,
x\mapsto(x,f(x)),
$
\(G\) acts trivially on the second factor and $i_f$ is \(G\)-equivariant, put
$
(\widetilde{\mathcal M},F):=(i_f)_+(\mathcal M,F).
$
Then $(\widetilde{\mathcal M},F)$ is a $G$-equivariant $\sD_{M'}$-module; there are isomorphisms $\phi_g:g^*(\widetilde{\mathcal M},F) \xrightarrow{\cong} (\widetilde{\mathcal M},F)$. \(t=\operatorname{pr}_2\) is \(G\)-equivariant, and every
\(g\in G\) preserves \(t\), \(\partial_t\). The canonical
Kashiwara--Malgrange \(V\)-filtration along $t$ is then \(G\)-equivariant \cite[Lemma 3.2]{LY}, i.e., $V_\alpha\widetilde{\mathcal M}$ is a $G$-equivariant $\mathcal{O}_{M'}$-submodule of $\widetilde{\mathcal M}$. This follows from the uniqueness of \(V\)-filtration, if
\(\{V_\alpha\widetilde{\mathcal M}\}\) is the canonical \(V\)-filtration,
then, for every \(g\),
$
{
g(V_\alpha\widetilde{\mathcal M}\bigr)
}
$
again satisfies all defining properties of the canonical \(V\)-filtration,
since \(g\) preserves \(t\) and \(\partial_t\). The Hodge filtration $F_{p}\widetilde{\mathcal M}$ is equivariant, therefore
the induced filtrations
\(F_pV_{\alpha}\widetilde{\mathcal M}=F_p \widetilde{\mathcal M} \cap V_{\alpha}\widetilde{\mathcal M}, \: F_pGr^V_\alpha\widetilde{{\mathcal M}}=F_pV_{\alpha}\widetilde{\mathcal M}/F_pV_{<\alpha}\widetilde{\mathcal M}\)
are also equivariant. All conditions of strict specializability \cite[Def. 2.4]{S2}
become equivariant, so all exact sequences and isomorphisms coming from
strict specializability
\cite[Lemmas.~2.1, 2.2]{S2} are exact sequences and isomorphisms in the
\(G\)-equivariant category. Since $G$ acts trivially on the second factor of $M'$, the splitting $\Omega^k_{M'}\cong \Omega^k_{M'/\mathbb{C}}\oplus dt \wedge \Omega^{k-1}_{M'/\mathbb{C}}$ is \(G\)-equivariant as well and all bifiltered double complexes are equivariant. Hence the equalities obtained there by
additivity hold in \(G_0^G(X_0)\).

\end{proof}

We mention the equivariant version of a result of M. Banagl on the specialization of the intersection Hodge module of the deformation to the normal cone \cite[Prop.~6.20]{B}.
\begin{theorem}\label{banagl_ic_nomralcone_z} Let $i:Y\hookrightarrow X$ be a $G$-equivariant  upwardly normally nonsingular embedding of  pure-dimensional compact complex algebraic varieties $X,Y$ (Definition \ref{upward}) with the  deformation  to the normal cone 
$
p:\mathcal{Z}^{\circ}\longrightarrow \mathbb{C}$ and the algebraic normal bundle $p^{-1}(0)=N$. Then in $D^{b,G}(\MHM(N))$
\[
\Psi_p^{H,G}(IC_{\mathcal{Z}^\circ}^H)\cong IC_N^H
\]
\end{theorem}
\begin{proof}
  We make some remarks on how some of the inputs in the proof of \cite[Prop.~6.20]{B} can be seen equivariantly. \cite[Lemma~6.19]{B} implies that there is an isomorphism $\alpha:i^*IC_X^H[-c]\xrightarrow{\cong}IC_Y^H$ in $D^b(\mhm(Y))$, where $c$ is the complex codimension of $Y/X$. Since we do not assume that the algebraic stratifications appearing in the upwardly normally nonsingular condition are $G$-invariant, the proof given in loc. cit. does not extend equivariantly in an immediate way. Instead, we construct an isomorphism in $D^{b,G}(\mhm(Y))$ as follows. Let $U\xrightarrow{j}X$ and $V=U\cap Y\xrightarrow{j_V}Y$ be the Zariski-open  dense smooth algebraic subvarieties of $X$ and $Y$ appearing in loc. cit. Set $U_G:=\bigcap_{h\in G}h.U$ and $V_G:=\bigcap_{h\in G}h.V$. Since $G$ is finite, $U_G$ and $V_G$ are again smooth dense Zariski-open subsets, and they are $G$-invariant. Since $i$ is $G$-equivariant, $V_G=U_G\cap Y$. Let $j_{V_G}:V_G\hookrightarrow Y$ denote the $G$-equivariant inclusion.  Since $IC_Y^H$ and $i^*IC_X^H[-c]$ have $G$-equivariant structure, there are canonical $G$-equivariant identifications $IC_Y^H|_{V_G}\cong \mathbb{Q}_{V_G}^H[d_Y]$ and $\left(i^*IC_X^H[-c]\right)|_{V_G}\cong\mathbb{Q}_{V_G}^H[d_Y]$. Under these identifications, the restriction of $\alpha$ to $V_G$ determines an automorphism $a:\mathbb{Q}_{V_G}^H[d_Y]\xrightarrow{\cong}\mathbb{Q}_{V_G}^H[d_Y]$. By functoriality of the intermediate extension, $a^{-1}$ induces an automorphism $j_{V_G!*}(a^{-1}):IC_Y^H\xrightarrow{\cong}IC_Y^H$. Replacing $\alpha$ by $\widetilde{\alpha}:=j_{V_G!*}(a^{-1})\circ\alpha$, we may therefore assume that, under the canonical identifications above, $\widetilde{\alpha}|_{V_G}=\operatorname{id}_{\mathbb{Q}_{V_G}^H[d_Y]}$. Then $\widetilde{\alpha}$ is $G$-equivariant. Indeed, for every $g\in G$, the $g$-translate $g\cdot\widetilde{\alpha}:=\psi^{-1}_g\circ g_*(\widetilde{\alpha})\circ \psi_g$ is another isomorphism $i^*IC_X^H[-c]\xrightarrow{\cong}IC_Y^H$ whose restriction to $V_G$ is the identity. Hence $(g\cdot\widetilde{\alpha})\widetilde{\alpha}^{-1}$ is an automorphism of $IC_Y^H$ which restricts to the identity on $\mathbb{Q}_{V_G}^H[d_Y]$. Since $IC_Y^H$ has no nonzero subobject supported on $Y\setminus V_G$, such an automorphism must be the identity. Thus $g\cdot\widetilde{\alpha}=\widetilde{\alpha}$ for every $g\in G$, and therefore $\widetilde{\alpha}$ is $G$-equivariant. Another input is \cite[Lemma~6.13]{B}, which says that the smooth pullback of (shifted) $IC^H$ is $IC^H$. The proof is similar to that of \cite[Lemma~6.19]{B}, we may replace $U,V$ by their invariant subsets $U_G,V_G$, and all the functors are now equivariant, which makes the isomorphisms hold in the equivariant category. Another input is a result of Saito that the specialization functor from mixed Hodge modules on $X$ to those on $N$ induces the identity on $Y$, \cite[Eq.~(12)]{B}. This is induced by a natural morphism of functors that are all $G$-equivariant, since all the maps in the deformation to the normal cone are \(G\)-equivariant. Since the forgetful functor is conservative, this implies that the equivariant map is also an isomorphism. 
\end{proof}


\begin{thebibliography}{99}

\bibitem[ABP1]{ABP1}
P.~Albin, M.~Banagl, and P.~Piazza,
\emph{Smooth atlas stratified spaces, $K$-homology orientations, and Gysin maps},
preprint, 2025, arXiv:2505.14952.

\bibitem[ABP2]{ABP2}
P.~Albin, M.~Banagl, and P.~Piazza,
\emph{Smooth atlas stratified spaces, $K$-homology orientations, and Gysin maps. Part 2},
preprint, 2026, arXiv:2605.28620.

\bibitem[ALMP]{ALMP}
P.~Albin, E.~Leichtnam, R.~Mazzeo, and P.~Piazza,
\emph{The signature package on Witt spaces},
Ann. Sci. \'Ec. Norm. Sup\'er. (4) \textbf{45} (2012), no.~2, 241--310.

\bibitem[AMS]{AMS}
M.~Aligholi, L.~Maxim, and J.~Sch\"urmann,
\emph{Around a class version of the Hodge index theorem for singular varieties},
preprint, arXiv:2511.08780 [math.AG].

\bibitem[AS2]{AS2}
M.~F.~Atiyah and G.~B.~Segal,
\emph{The index of elliptic operators: II},
Ann. of Math. (2) \textbf{87} (1968), no.~3, 531--545.

\bibitem[AS3]{AS3}
M.~F.~Atiyah and I.~M.~Singer,
\emph{The index of elliptic operators: III},
Ann. of Math. (2) \textbf{87} (1968), no.~3, 546--604.

\bibitem[B]{B}
M.~Banagl,
\emph{Gysin restriction of topological and Hodge-theoretic characteristic classes for singular spaces},
New York J. Math. \textbf{26} (2020), 1273--1337.

\bibitem[B26]{B26}
M.~Banagl,
\emph{Equivariant $L$-classes of Atiyah--Singer--Zagier type for singular spaces},
Bull. Sci. Math. \textbf{213} (2026), Paper No.~103885.

\bibitem[BLP]{BLP}
M.~Banagl, E.~Leichtnam, and P.~Piazza,
\emph{The $G$-signature theorem on Witt spaces},
preprint, 2025, arXiv:2511.21586 [math.DG].

\bibitem[BD]{BD}
P.~Baum and R.~G.~Douglas,
\emph{$K$-homology and index theory},
in \emph{Operator Algebras and Applications, Part I},
Proc. Sympos. Pure Math., vol.~38,
Amer. Math. Soc., Providence, RI, 1982, pp.~117--173.

\bibitem[BE]{BE}
P.~Baum and E.~van Erp,
\emph{$K$-homology and index theory on contact manifolds},
Acta Math. \textbf{213} (2014), no.~1, 1--48.

\bibitem[BFM]{BFM}
P.~Baum, W.~Fulton, and R.~MacPherson,
\emph{Riemann--Roch for singular varieties},
Publ. Math. Inst. Hautes \'Etudes Sci. \textbf{45} (1975), 101--145.

\bibitem[BFM2]{BFM2}
P.~Baum, W.~Fulton, and R.~MacPherson,
\emph{Riemann--Roch and topological $K$-theory for singular varieties},
Acta Math. \textbf{143} (1979), no.~3--4, 155--192.

\bibitem[BFQ]{BFQ}
P.~Baum, W.~Fulton, and G.~Quart,
\emph{Lefschetz--Riemann--Roch for singular varieties},
Acta Math. \textbf{143} (1979), 193--211.

\bibitem[BHS]{BHS}
P.~Baum, N.~Higson, and T.~Schick,
\emph{On the equivalence of geometric and analytic $K$-homology},
Pure Appl. Math. Q. \textbf{3} (2007), no.~1, 1--24.

\bibitem[BOSW]{BOSW}
P.~Baum, H.~Oyono-Oyono, T.~Schick, and M.~Walter,
\emph{Equivariant geometric $K$-homology for compact Lie group actions},
Abh. Math. Semin. Univ. Hambg. \textbf{80} (2010), 149--173.

\bibitem[BSY]{BSY}
J.-P.~Brasselet, J.~Sch\"urmann, and S.~Yokura,
\emph{Hirzebruch classes and motivic Chern classes for singular spaces},
J. Topol. Anal. \textbf{2} (2010), no.~1, 1--55.

\bibitem[CMSS]{CMSS}
S.~E.~Cappell, L.~Maxim, J.~Sch\"urmann, and J.~L.~Shaneson,
\emph{Equivariant characteristic classes of singular complex algebraic varieties},
Comm. Pure Appl. Math. \textbf{65} (2012), no.~12, 1722--1769.

\bibitem[CS]{CS}
S.~E.~Cappell and J.~L.~Shaneson,
\emph{Stratifiable maps and topological invariants},
J. Amer. Math. Soc. \textbf{4} (1991), no.~3, 521--551.

\bibitem[Ch]{Ch}
J.~Cheeger,
\emph{Spectral geometry of singular Riemannian spaces},
J. Differential Geom. \textbf{18} (1983), no.~4, 575--657.

\bibitem[FP]{FP}
J.~Fern\'andez de Bobadilla and I.~Pallar\'es,
\emph{The Brasselet--Sch\"urmann--Yokura conjecture on $L$-classes of projective rational homology manifolds},
Int. Math. Res. Not. IMRN \textbf{2024} (2024), no.~19, 13085--13105.

\bibitem[FPS]{FPS}
J.~Fern\'andez de Bobadilla, I.~Pallar\'es, and M.~Saito,
\emph{Hodge modules and cobordism classes},
J. Eur. Math. Soc. \textbf{27} (2025), no.~2, 773--800.

\bibitem[GM]{GM}
M.~Goresky and R.~MacPherson,
\emph{Intersection homology theory},
Topology \textbf{19} (1980), 135--162.

\bibitem[HZ]{HZ}
F.~Hirzebruch and D.~Zagier,
\emph{The Atiyah--Singer Theorem and Elementary Number Theory},
Mathematics Lecture Series, vol.~3,
Publish or Perish, Boston, 1974.

\bibitem[J]{J}
M.~Jakob,
\emph{A bordism-type description of homology},
Manuscripta Math. \textbf{96} (1998), no.~1, 67--80.

\bibitem[K]{K}
G.~G.~Kasparov,
\emph{The operator $K$-functor and extensions of $C^*$-algebras},
Math. USSR-Izv. \textbf{16} (1981), no.~3, 513--572.

\bibitem[LY]{LY}
A.~C.~L{\H{o}}rincz and R.~Yang,
\emph{Filtrations of $\mathscr D$-modules along semi-invariant functions},
Adv. Math. \textbf{499} (2026), 111037.

\bibitem[MS]{MS}
L.~Maxim and J.~Sch\"urmann,
\emph{Characteristic classes of mixed Hodge modules and applications},
in \emph{Schubert Varieties, Equivariant Cohomology and Characteristic Classes},
EMS Ser. Congr. Rep.,
European Mathematical Society, Z\"urich, 2017, pp.~159--202.



\bibitem[M]{M}
B.~Moonen,
\emph{Das Lefschetz--Riemann--Roch Theorem f\"ur singul\"are Variet\"aten},
Bonner Mathematische Schriften, vol.~106,
Bonn, 1978.





\bibitem[MW]{MW}
H.~Moscovici and F.~Wu,
\emph{Straight Chern character for Witt spaces},
in \emph{Cyclic Cohomology and Noncommutative Geometry
(Waterloo, ON, 1995)},
Fields Inst. Commun., vol.~17,
Amer. Math. Soc., Providence, RI, 1997, pp.~103--113.

\bibitem[P]{P}
M.~Puschnigg,
\emph{Diffeotopy functors of ind-algebras and local cyclic cohomology},
Doc. Math. \textbf{8} (2003), 143--245.

\bibitem[Q]{Q}
F.~Qu,
\emph{Virtual pullbacks in $K$-theory},
Ann. Inst. Fourier (Grenoble) \textbf{68} (2018), no.~4, 1609--1641.

\bibitem[Q79]{Q79}
G.~Quart,
\emph{Localization theorem in $K$-theory for singular varieties},
Acta Math. \textbf{143} (1979), 213--217.

\bibitem[R]{R}
J.~Rosenberg,
\emph{The $G$-signature theorem revisited},
in \emph{Tel Aviv Topology Conference: Rothenberg Festschrift},
Contemp. Math., vol.~231,
Amer. Math. Soc., Providence, RI, 1999, pp.~251--264.

\bibitem[RW]{RW}
J.~Rosenberg and S.~Weinberger,
\emph{Higher $G$-signatures for Lipschitz manifolds},
K-Theory \textbf{7} (1993), no.~2, 101--132.

\bibitem[RW2]{RW2}
J.~Rosenberg and S.~Weinberger,
\emph{The signature operator at 2},
Topology \textbf{45} (2006), no.~1, 47--63.

\bibitem[Sa]{Sa}
M.~Saito,
\emph{Mixed Hodge modules},
Publ. Res. Inst. Math. Sci. \textbf{26} (1990), 221--333.

\bibitem[Sch]{Sch}
C.~Schochet,
\emph{On equivariant Kasparov theory and Spanier--Whitehead duality},
K-Theory \textbf{6} (1992), no.~4, 363--385.

\bibitem[S]{S}
J.~Sch\"urmann,
\emph{Characteristic classes of mixed Hodge modules},
in \emph{Topology of Stratified Spaces},
Math. Sci. Res. Inst. Publ., vol.~58,
Cambridge University Press, Cambridge, 2011, pp.~419--470.

\bibitem[S2]{S2}
J.~Sch\"urmann,
\emph{Specialization of motivic Hodge--Chern classes},
preprint, 2009, arXiv:0909.3478.

\bibitem[Ta]{Ta}
Y.~Takeda,
\emph{Localization theorem in equivariant algebraic $K$-theory},
J. Pure Appl. Algebra \textbf{96} (1994), no.~1, 73--80.

\bibitem[T2]{T2}
R.~W.~Thomason,
\emph{Equivariant algebraic vs.\ topological $K$-homology Atiyah--Segal-style},
Duke Math. J. \textbf{56} (1988), no.~3, 589--636.

\bibitem[T1]{T1}
R.~W.~Thomason,
\emph{Une formule de Lefschetz en $K$-th\'eorie \'equivariante alg\'ebrique},
Duke Math. J. \textbf{68} (1992), no.~3, 447--462.

\bibitem[Z]{Z}
D.~Zagier,
\emph{The Pontryagin class of an orbit space},
Topology \textbf{11} (1972), no.~3, 253--264.

\end{thebibliography}
\end{document}